\documentclass[12pt,a4paper, oneside]{article}
\usepackage[a4paper,top=3cm,bottom=3cm,
left=3cm,right=3cm,%
 bindingoffset=0mm]{geometry}

\usepackage{graphicx} 
\usepackage{mathptmx}      
\usepackage{latexsym}
\usepackage{amsmath,amssymb}
\usepackage{comment}
\usepackage{algorithmic}
\usepackage{algorithm} 
\usepackage{bm}
\usepackage{enumerate, color}
\usepackage[mathscr]{eucal} 
\usepackage{eso-pic}
\usepackage[applemac]{inputenc}
\usepackage[T1]{fontenc}
\usepackage{fancyhdr}

\usepackage{ragged2e} 
\usepackage[english]{babel}
\usepackage{times}
\usepackage[T1]{fontenc}
\usepackage{amssymb}
\usepackage{amsfonts}
\usepackage{amsmath,amssymb}
\usepackage{amsthm}
\usepackage{rotating}
\usepackage{enumerate}
\usepackage{eso-pic,graphicx}
\usepackage{caption} 
\usepackage{subfig}
\usepackage{epstopdf}
\usepackage{eucal} 
\usepackage{verbatim}
\usepackage{times}
\usepackage{amssymb}
\usepackage{amsfonts}
\usepackage{amsmath,amssymb}
\usepackage{amsthm}
\usepackage{ragged2e}
\usepackage{colortbl}
\usepackage{algorithm}
\usepackage{algorithmic}

\usepackage{listings}

\usepackage{float}
\usepackage{enumerate}

\newtheorem{theorem}{Theorem}[section]
\newtheorem{corollary}{Corollary}[theorem]
\newtheorem{lemma}[theorem]{Lemma}
\newtheorem{proposition}[theorem]{Proposition}

\title{$\gamma$-matrices: a new class of simultaneously diagonalizable matrices}
\author{Antonio Boccuto, Ivan Gerace, Valentina Giorgetti
and Federico Greco\\ \\
Laboratorio di Matematica Computazionale  
``Sauro Tulipani''\\
Dipartimento di Matematica e Informatica\\
Universit\`a degli Studi di Perugia\\
Via Vanvitelli, 1  \\
I-06123 Perugia (Italy)\\ \\
\includegraphics[height=0.15\textwidth]{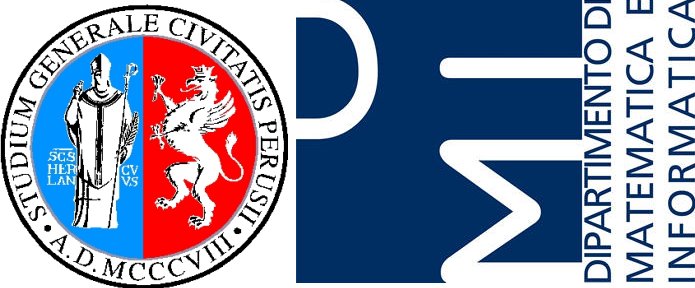}
\includegraphics[width=0.15\textwidth]{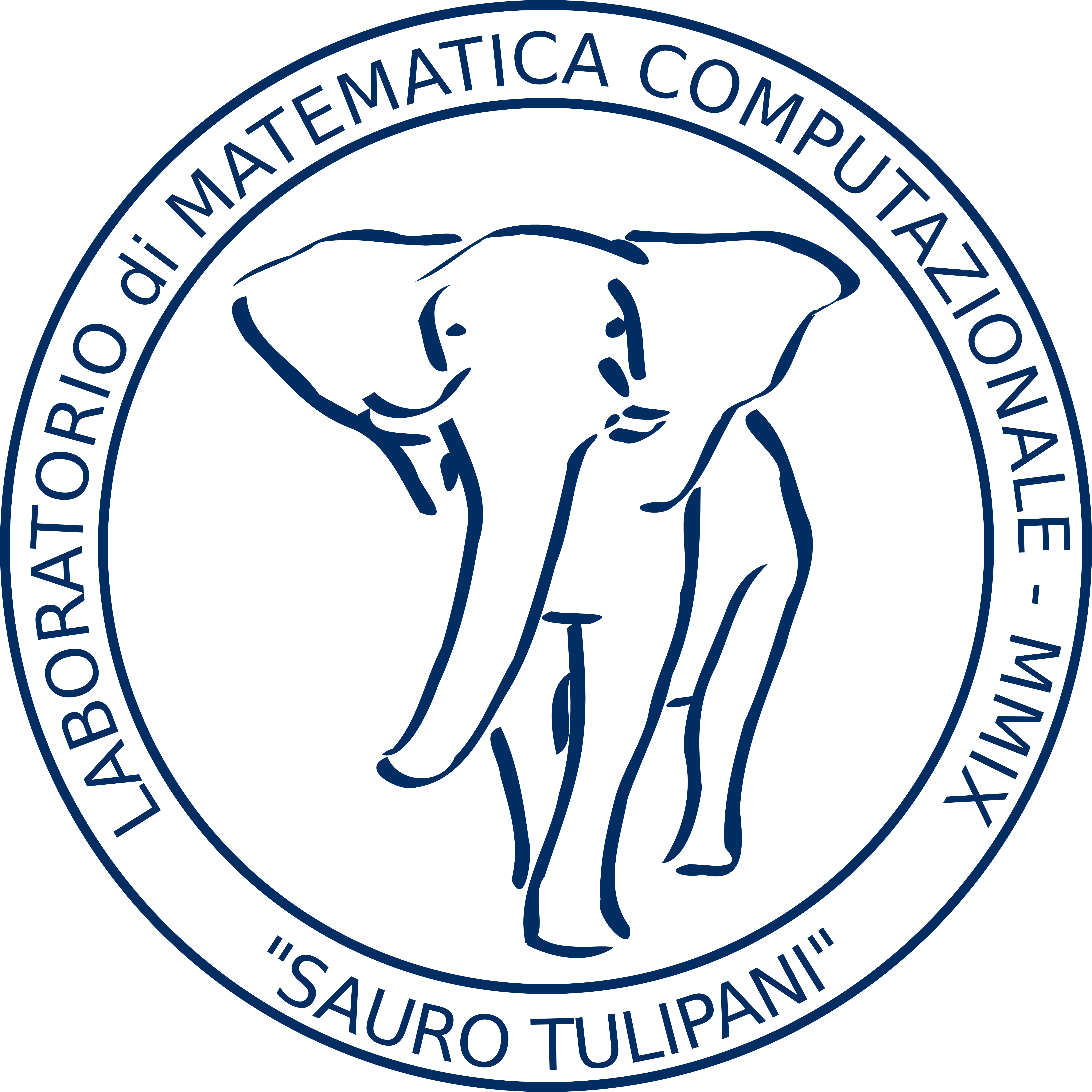}}
\begin{document}
\maketitle


\begin{abstract} 
In order to precondition Toeplitz systems, 
we present a new class of 
simultaneously diagonalizable real matrices,
the $\gamma$-ma\-tri\-ces, which
include both symmetric circulant matrices and a subclass of the set 
of all reverse circulant matrices.
We define some algorithms 
for fast computation
of the product between a $\gamma$-matrix
and a real vector 
and between two $\gamma$-matrices. 
Moreover, we illustrate a technique of approximating a real symmetric
Toeplitz matrix by a $\gamma$-matrix, and we show that
the eigenvalues of the preconditioned matrix are clustered 
around zero with the exception of at most a finite number 
of terms.
\end{abstract}

\section{Introduction}
\label{intro}
Toeplitz-type linear systems arise from numerical approximation of 
differential equations (see, e.g., \cite{bertacciniomega}, 
\cite{DISCEPOLI}). Moreover,
in restoration of blurred images, it is often dealt with 
Toeplitz matrices (see, e.g., \cite{bgm}, \cite{bgp},
\cite{dellacqua}, \cite{donatelliserra}).
The related linear system is often ill-conditioned. Thus, 
preconditioning the system turns out be very useful to 
have stability of the result. In particular, if the linear 
system is solved by means of the 
conjugate gradient method, the preconditioning technique
allows to grow the speed of convergence.

Recently, there have been several studies about simultaneously diagonalizable real matrices.
In particular, in investigating the preconditioning of 
Toeplitz matrices (see, e.g., \cite{DFZ}), 
several matrix algebras 
are considered, among which
the class of all circulant matrices
(see, e.g., \cite{bertaccinicirc}, \cite{CHAN}, \cite{CHANetal}, 
\cite{CHANSTRANG}, \cite{DYKES}, \cite{gpcdm},
\cite{xiao}, \cite{osele}), the family of 
all $\omega$-circulant matrices (see, e.g., 
\cite{bertacciniomega}, \cite{FISCHER}),
the set of all 
$\tau$-matrices 
(see, e.g., \cite{BINIDIBENEDETTO}) and
the family of all matrices
diagonalized by the Hartley transform
(see, e.g., \cite{arico}, \cite{BINIFAVATI1993}, \cite{bortol}).

In this paper we investigate a particular class of 
simultaneously diagonalizable real matrices,
whose elements we call \emph{$\gamma$-matrices}. Such a class,
similarly as that of the matrices diagonalizable by 
the Hartley transform (see, e.g., \cite{BINIFAVATI1993})
includes symmetric circulant matrices and a subclass of
the family of all reverse circulant matrices.
Symmetric circulant matrices have several applications 
to ordinary and partial differential equations 
(see, e.g., \cite{evans}, \cite{gilmour},
\cite{gyori}, \cite{gyori2}), images and signal restoration 
(see, e.g., \cite{carras}, \cite{henriques}), 
graph theory (see, e.g., \cite{CGV},
\cite{DISCEPOLI}, \cite{geracegreco}, \cite{gg},
\cite{gutekunst},
\cite{gutekunst2021}).
Reverse circulant matrices have different applications,
for instance in exponential data fitting and signal processing 
(see, e.g., \cite{andrecut},
\cite{1}, \cite{4}, \cite{16}, \cite{18}).

We first analyze both spectral and structural
properties of this family of matrices. 
Successively, we deal with the problem of real
fast transform algorithms 
(see, e.g., \cite{ahmed}, \cite{arico}, \cite{bergland},
\cite{duhamel}, \cite{ersoy},
\cite{gupta}, \cite{heinig}, 
\cite{martens}, \cite{murakami},
\cite{SCARABOTTI}, \cite{strang}). In particular
we define some algorithms for fast computation
of the product between a $\gamma$-matrix
and a real vector 
and between two $\gamma$-matrices. 
We show that the computational cost of a multiplication
between a $\gamma$-matrix and a real vector is of at most
$\frac74 \, n \, \log_2 n+o( n \, \log_2 n)$ additions and
$\frac12 \, n \, \log_2 n+o( n \, \log_2 n)$ multiplications, 
while the computational cost of a multiplication
between two $\gamma$-matrices is at most
$\frac92 \, n \, \log_2 n+o( n \, \log_2 n)$ additions and
$\frac32 \, n \, \log_2 n+o( n \, \log_2 n)$ multiplications.
Finally, we present a technique of approximating a real symmetric
Toeplitz matrix by a $\gamma$-matrix, and we show that
the eigenvalues of the preconditioned matrix are clustered 
around zero with the exception of at most a finite number 
of terms.

In Section 2 we investigate spectral properties of 
$\gamma$-matrices, in Section 3 we deal with 
structural properties, in Section 4 we define the real
fast transforms by means of which it is possible to
implement the product between 
a $\gamma$-matrix and a real vector.
In Section 5 we use the previously defined transforms
to implement the multiplication between two
$\gamma$-matrices. In Section 6 we deal with
the problem of preconditioning a real symmetric
Toeplitz matrix by a $\gamma$-matrix.

%
%

\section{Spectral characterization of $\gamma$-matrices} 

To present a new class of simultaneously
diagonalizable matrices, we first define the following 
matrix. Let $n$ be a fixed positive integer, and
$Q_n=(q^{(n)}_{k,j})_{k,j}$, $k$, $j=0$, $1, \ldots, n-1$, where

\begin{eqnarray}\label{prestartingmatrix}
q^{(n)}_{k,j}=\left\{ \begin{array}{ll}
\alpha_j \, \cos \Bigl( \dfrac{2 \pi \, k \,  j}{n} \Bigr) & \text{if  } 
0 \leq j \leq \lfloor n/2 \rfloor, \\  \\ \alpha_j \,
\sin \Bigl( \dfrac{2 \pi \, k \, (n-j)}{n} \Bigr) & \text{if  }
\lfloor n/2 \rfloor \leq j \leq n-1 ,
\end{array} \right.
\end{eqnarray}
\begin{eqnarray}\label{alphak}
\alpha_j= \left\{
\begin{array}{ll}
\dfrac{1}{\sqrt{n}}=\overline{\alpha}
 & \text{if  } j=0, \text{  or   }j=n/2 
\text{   if   }n \text{   is  even,  }
\\  \\\sqrt{\dfrac{2}{n}}= \widetilde{\alpha} & \text{otherwise,}
\end{array} \right. \, \, 
\end{eqnarray}
and put  
\begin{eqnarray}\label{qn}
Q_n=
\Bigl(  {\mathbf{q}}^{(0)} \, \Bigr| \, 
{\mathbf{q}}^{(1)} \Bigl|  \, \cdots \, \Bigr| \, 
{\mathbf{q}}^{(\lfloor \frac{n}{2} \rfloor)} \, \Bigl| \,
{\mathbf{q}}^{(\lfloor \frac{n+1}{2} \rfloor)} \, \Bigr| \, 
\, \cdots \, \Bigl| \, {\mathbf{q}}^{(n-2)} \, \Bigr|  {\mathbf{q}}^{(n-1)}
\Bigr),
\end{eqnarray}
where \begin{eqnarray}\label{35}   {\mathbf{q}}^{(0)}= \dfrac{1}{\sqrt{n}} \, \Bigl( 1 \, \, 
1 \, \, \cdots \, \, 1 \Bigr)^T =\dfrac{1}{\sqrt{n}} \,
{\mathbf{u}}^{(0)},
\end{eqnarray}
\begin{eqnarray}\label{yjzj} {\mathbf{q}}^{(j)}&=& \sqrt{\dfrac{2}{n}} \, \left(   1 \, \, \, \, 
\cos \left( \dfrac{2 \pi j}{n}\right) \, 
\cdots \,  \cos \left( \dfrac{2\pi j(n-1)}{n}\right) \, 
\right)^T =\sqrt{\dfrac{2}{n}} \,
{\mathbf{u}}^{(j)},  \nonumber \\  {\mathbf{q}}^{(n-j)}
 &=&
\sqrt{\dfrac{2}{n}} \, \left(   0  \, \,
\sin \left( \dfrac{2 \pi j}{n}\right) \, 
\cdots \, \sin \left( \dfrac{2 \pi j(n-1)}{n}\right) \, 
\right)^T =\sqrt{\dfrac{2}{n}} \,
{\mathbf{v}}^{(j)},  
\end{eqnarray} $j=1$, $2, \ldots, \lfloor \frac{n-1}{2} \rfloor$.
Moreover, when $n$ is even, 
set \begin{eqnarray}\label{45}
 {\mathbf{q}}^{(n/2)}=\dfrac{1}{\sqrt{n}}
\Bigl( 1  \, -1 \, \, \, \, 1 \, -1 \, \cdots \, -1\Bigr)^T=
\dfrac{1}{\sqrt{n}} \,
{\mathbf{u}}^{(n/2)}.
\end{eqnarray}
In \cite{CINESI} it is proved that
all columns of $Q_n$ are orthonormal,
and thus $Q_n$ is an orthonormal matrix.

Now we define the following function.
Given 
$\bm{\lambda} \in \mathbb{C}^n$,
 $\bm{\lambda}= (\lambda_0 \, \lambda_1 \cdots 
\lambda_{n-1})^T$, set 
\begin{eqnarray*}
\text{diag}(\bm{\lambda})=\Lambda=\begin{pmatrix} 
\lambda_0 & 0 & 0 &\ldots & 0 & 0 \\
0 & \lambda_1 & 0 &\ldots & 0 & 0 \\
0 & 0 & \lambda_2 &\ldots & 0 & 0\\
\vdots & \vdots & \vdots &\ddots & \vdots & \vdots \\
0 & 0 & 0 &\ldots & \lambda_{n-2} & 0 \\
0 & 0 & 0 &\ldots & 0 & \lambda_{n-1}
\end{pmatrix},
\end{eqnarray*}
where $\Lambda \in {\mathbb{C}}^{n \times n}$ is a diagonal matrix.

A vector $\bm{\lambda} \in \mathbb{R}^n$, $\bm{\lambda}= (\lambda_0 \, \lambda_1 \cdots 
\lambda_{n-1})^T$ is said to be \emph{symmetric}
(resp., \emph{asymmetric}) iff 
$\lambda_j=\lambda_{n-j}$ (resp., $\lambda_j=-\lambda_{n-j}) \in 
\mathbb{R}$ for every $j=0$, $1,\ldots, \lfloor n/2 \rfloor$.

Let $Q_n$ be as in (\ref{qn}), and
${\mathcal G}_n$ be the space of the matrices
\emph{simultaneously diagonalizable} by $Q_n$, that is 
\begin{eqnarray*}
{\mathcal G}_n= {\rm{sd  } } (Q_n) = \{ Q_n \Lambda Q_n^T: \Lambda 
=\text{diag}(\bm{\lambda}),  \, \bm{\lambda} \in \mathbb{R}^n\}.
\end{eqnarray*}

A matrix belonging to ${\mathcal G}_n$, $n \in \mathbb{N}$,
is called
\emph{$\gamma$-matrix}.
Moreover, we define the following classes by 
\begin{eqnarray*}\label{Sn}
{\mathcal C}_n = \{ Q_n \Lambda Q_n^T: \Lambda 
=\text{diag}(\bm{\lambda}), 
\, \bm{\lambda} \in \mathbb{R}^n, \, \bm{\lambda} \text{  is  symmetric}\},
\end{eqnarray*}
\begin{eqnarray*}\label{Bn}
{\mathcal B}_n = \{ Q_n \Lambda Q_n^T: \Lambda 
=\text{diag}(\bm{\lambda}), \, \bm{\lambda} \in \mathbb{R}^n,\, \bm{\lambda} \text{  is  asymmetric}\},
\end{eqnarray*}
\begin{eqnarray*}\label{Dn}
{\mathcal D}_n &=& \{ Q_n \Lambda Q_n^T: \Lambda 
=\text{diag}(\bm{\lambda}), \, \bm{\lambda} \in \mathbb{R}^n,
\, \bm{\lambda} \text{  is  symmetric}, \\ && 
\lambda_0=0, \lambda_{n/2}=0 \text{    if  } n \text{  is  even}\},
\end{eqnarray*}
\begin{eqnarray*}\label{En}
{\mathcal E}_n &=& \{ Q_n \Lambda Q_n^T: \Lambda 
=\text{diag}(\bm{\lambda}), \, \, \bm{\lambda} \in \mathbb{R}^n, \lambda_j=0,
\, \, j=1,\ldots, n-1, \\ && 
j \neq n/2 \text{    when  } n \text{  is  even}\}.
\end{eqnarray*}
\begin{proposition}\label{Proposition3.4}
The class ${\mathcal G}_n$ is a matrix algebra of dimension $n$.
\end{proposition}
\begin{proof} 
We prove that ${\mathcal G}_n$ is an algebra.
Let $I_n$ be the identity $n \times n$-matrix. 
Since $Q_n$
is orthogonal, then
$Q_n I_n Q_n^T=I_n$. Hence, $I_n \in {\mathcal G}_n$.

If $C \in {\mathcal G}_n$, $C$ is non-singular, $C=Q_n\Lambda Q_n^T$ and
$\Lambda$ is diagonal,
then $C^{-1}=
Q_n \Lambda^{-1} Q_n^T$, and hence
$C^{-1} \in {\mathcal G}_n$, since $\Lambda^{-1}$ 
is diagonal. 

Moreover, if 
$C_r \in {\mathcal G}_n$, $\alpha_r \in \mathbb{R}$,
$C_r=Q_n  
\Lambda_r Q_n^T$ and
$\Lambda_r$
is diagonal,  $r=1,2$, then 
$\alpha_1 C_1+ \alpha_2
C_2= Q_n (\alpha_1 \Lambda_1 + 
\alpha_2 \Lambda_2) Q_n^T
\in {\mathcal G}_n$, since 
$\alpha_1 \Lambda_1 + \alpha_2 \Lambda_2$ 
is diagonal. Furthermore,
$C_1 C_2 = Q_n  \Lambda_1 Q_n^T 
Q_n  \Lambda_2 Q_n^T=Q_n  \Lambda_1  
\Lambda_2 Q_n^T$, since $\Lambda_1 \Lambda_2$ 
is diagonal. Therefore, ${\mathcal G}_n$ is an algebra.

Now we claim that $\dim({\mathcal G}_n)=\dim(\bm{\lambda})=n$. 
By contradiction, let $\bm{\lambda_1}\neq\bm{\lambda_2} \in 
{\mathbb{R}}^n$ be such that $Q_n$diag$(\bm{\lambda_1})Q_n^T=
Q_n$diag$(\bm{\lambda_2})Q_n^T=C$. Then, the elements of 
$\bm{\lambda_2}$ are obtained by a suitable permutation of those of
$\bm{\lambda_1}$.
Since the order of the
eigenvectors of $C$ have been established, if a component 
$\lambda_j^{(1)}$ of $\bm{\lambda_1}$ is equal to a component
$\lambda_k^{(2)}$ of $\bm{\lambda_2}$, then $\mathbf{q}^{(j)}$ 
and $\mathbf{q}^{(k)}$ belong to the same eigenspace,
and hence $\lambda_j^{(1)}=\lambda_j^{(2)}=
\lambda_k^{(1)}=\lambda_k^{(2)}$. This implies that 
$\bm{\lambda_1}=\bm{\lambda_2}$, a contradiction.
This ends the proof.
  
\end{proof}
\begin{proposition}\label{insert} 
The class ${\mathcal C}_n$ is a subalgebra of ${\mathcal G}_n$
of dimension $\lfloor \frac{n}{2} \rfloor+1$.
\end{proposition}
\begin{proof} 
Obviously, ${\mathcal C}_n \subset {\mathcal G}_n$.
Now we claim that ${\mathcal C}_n$ is an algebra. First, note that
$I_n \in {\mathcal C}_n$, since 
$I_n=Q_n I_n Q_n^T$ and $I_n=$diag$(\mathbf{1})$,
where $\mathbf{1}=( 1 \, \, 1 \cdots 1)^T$.
So, $I_n \in {\mathcal C}_n$, since $\mathbf{1}$ is symmetric.

If $C \in {\mathcal C}_n$, $C$ is non-singular, $C=Q_n\Lambda Q_n^T$,
$\Lambda=$diag$(\bm{\lambda})$, $\bm{\lambda}=$ 
$(\lambda_0 \, \, \lambda_1 \cdots \lambda_{n-1})^T$ is symmetric,
then $C^{-1}=
Q_n \Lambda^{-1} Q_n^T$, and so
$C^{-1} \in {\mathcal C}_n$, as $\Lambda^{-1}
=$diag$(\bm{\lambda}^{\prime})$, and
$\bm{\lambda}^{\prime}=$ 
$( 1/\lambda_0 \, \, 1/\lambda_1 \cdots 1/\lambda_{n-1})^T$ is symmetric too. 

If 
$C_r \in {\mathcal C}_n$, $\alpha_r \in \mathbb{R}$,
$C_r=Q_n  
\Lambda_r Q_n^T$,
$\Lambda_r=$diag$(\bm{\lambda^{(r)}})$, $\bm{\lambda}^{(r)}=$ 
$(\lambda^{(r)}_0 \, \, \lambda^{(r)}_1 \cdots
\lambda^{(r)}_{n-1})^T$ is symmetric, $r=1,2$, then 
$\alpha_1 C_1+ \alpha_2
C_2= Q_n (\alpha_1 \Lambda_1 + 
\alpha_2 \Lambda_2) Q_n^T
\in {\mathcal C}_n$, since 
$\alpha_1 \Lambda_1 + \alpha_2 \Lambda_2=$diag$({\bm{\lambda}^*})$, and
${\bm{\lambda}^*}$=
$(\alpha_1 \lambda^{(1)}_0+\alpha_2 \lambda^{(2)}_0$ $\cdots$ 
$\alpha_1 \lambda^{(1)}_{n-1}+\alpha_2 \lambda^{(2)}_{n-1} )^T
$ is symmetric. Furthermore,
$C_1 C_2 = 
Q_n  \Lambda_1  
\Lambda_2 Q_n^T$, since 
$\Lambda_1 \Lambda_2=$diag$({\bm{\lambda}_*})$, and 
${\bm{\lambda}_*}=$ $(\lambda^{(1)}_0 \, \lambda^{(2)}_0
$ $\cdots$ $\lambda^{(1)}_{n-1} \, \lambda^{(2)}_{n-1})^T$
is symmetric. Therefore, ${\mathcal C}_n$ is an algebra.

Now we prove that $\dim({\mathcal C}_n)=\lfloor \frac{n}{2} \rfloor+1$.
By the definition of ${\mathcal C}_n$,
it is possible to choose at most $\lfloor \frac{n}{2} \rfloor+1$ 
elements of ${\bm{\lambda}}$. The proof is analogous to that of 
the last part of Proposition \ref{Proposition3.4}.
  \end{proof}
\begin{proposition}\label{bn}
The class ${\mathcal B}_n$ is a linear subspace 
of ${\mathcal G}_n$, and has
dimension $\lfloor \frac{n-1}{2}\rfloor$.
\end{proposition}
\begin{proof}
First, let us prove that ${\mathcal B}_n$ is a
linear subspace of ${\mathcal G}_n$. For $r=1,2$, let 
$A_r \in {\mathcal B}_n$, $\alpha_r \in \mathbb{R}$,
$\Lambda_r=\text{diag}({\bm{\lambda^{(r)}}})$, with 
$\bm{\lambda^{(r)}}$ asymmetric,
and $C_r=Q_n  
\Lambda_r Q_n^T$. Then 
$\alpha_1 C_1+ \alpha_2
C_2= Q_n (\alpha_1 \Lambda_1 + 
\alpha_2 \Lambda_2) Q_n^T$
with $\alpha_1 \Lambda_1 + \alpha_2 \Lambda_2
=$diag$({\bm{\lambda}^*})$, and
${\bm{\lambda}^*}$=
$(\alpha_1 \lambda^{(1)}_0+\alpha_2 \lambda^{(2)}_0$ $\cdots$ 
$\alpha_1 \lambda^{(1)}_{n-1}+\alpha_2 \lambda^{(2)}_{n-1} )^T
$ is asymmetric.
Therefore, $\alpha_1 C_1+ \alpha_2
C_2 \in {\mathcal B}_n$. So, ${\mathcal B}_n$ is a
linear subspace of ${\mathcal G}_n$.

Now we prove that $\dim({\mathcal B}_n)=\lfloor \frac{n-1}{2} \rfloor$.
By the definition of ${\mathcal B}_n$, 
it is possible to choose at most $\lfloor \frac{n-1}{2} \rfloor$ 
elements of ${\bm{\lambda}}$, because $\lambda_0=0$
and $\lambda_{n/2}=0$ when $n$ is even. The proof is analogous to that of 
the last part of Proposition \ref{Proposition3.4}.  
\end{proof}
Similarly as in Propositions \ref{Proposition3.4} and 
\ref{insert}, it is possible to prove that 
${\mathcal D}_n$ is a subalgebra
of ${\mathcal G}_n$ of
dimension $\lfloor \frac{n-1}{2}\rfloor$ and 
${\mathcal E}_n$ is a subalgebra of ${\mathcal G}_n$
of dimension $1$ when $n$ is odd and $2$ when $n$ is even.
Now we prove the following
\begin{theorem}\label{directsum}
One has
\begin{eqnarray}\label{fundamental1}
{\mathcal G}_n={\mathcal C}_n \oplus {\mathcal B}_n,
\end{eqnarray} 
where $\oplus$ is the orthogonal sum, and
$\langle \cdot , \cdot \rangle $ denotes the Frobenius product, defined by
\[
\langle G_1,G_2\rangle \, = \, \textrm{tr}(G_1^T G_2),
\qquad G_1, G_2\in {\mathcal G}_n,
\]
where $tr(G)$ is the trace of the 
matrix $G$.
\end{theorem}  
\begin{proof}
Observe that, to prove (\ref{fundamental1}),
it is enough to demonstrate the following properties:
\begin{description}
    \item[{\rm \ref{directsum}.1)}] ${\mathcal C}_n \cap
{\mathcal B}_n =\{O_n\}$,
where $O_n \in {\mathbb{R}}^{n \times n}$ is the matrix 
whose entries are equal to $0$;
    \item[{\rm \ref{directsum}.2)}]
for any $G\in{\mathcal G}_n$, there exist
$C\in {\mathcal C}_n$ and 
$B \in {\mathcal B}_n$ with 
$G=C + B$;
   \item[{\rm \ref{directsum}.3)}] 
for any $C\in {\mathcal C}_n$ and 
$B \in {\mathcal B}_n$, it is $C +B \in {\mathcal G}_n$;
   \item[{\rm \ref{directsum}.4)}] 
$\langle C,B\rangle=0$ for each 
$C\in{\mathcal C}_n$ and 
$B\in{\mathcal B}_n$.
\end{description}

\ref{directsum}.1) \, \, 
Let $G\in{\mathcal C}_n \cap {\mathcal B}_n$. 
Then, $G=Q_n \Lambda^{(G)} Q_n^T$, where
$\Lambda^{(G)}=\text{diag}({\bm{\lambda}}^{(G)})$ and 
${\bm{\lambda}}^{(G)}$ is both symmetric and asymmetric.
But this is possible if and only if
${\bm{\lambda}}^{(G)}={\bm{0}}$, where 
${\bm{0}}$ is the vector whose components are equal to $0$.
Thus, $\Lambda^{(G)}=O_n$ and hence $G=O_n$. This proves
\ref{directsum}.1).

\ref{directsum}.2) \, \, 
Let $G \in {\mathcal G}_n$, 
$\Lambda^{(G)} \in {\mathbb{R}}^{n \times n}$
be such that $G=Q_n \Lambda^{(G)} Q_n^T$, 
$\Lambda^{(G)}=
$diag$({\bm{\lambda}}^{(G)})=\text{diag}(\lambda_0^{(G)}$
$\lambda_1^{(G)} \,
\cdots \, \lambda_{n-1}^{(G)})$. For $j \in \{0$, $1, \ldots, n-1\}$,
set \begin{eqnarray*}
\lambda^{(C)}_j=\dfrac{\lambda_j^{(G)}+
\lambda^{(G)}_{(n-j) \bmod
n}}{2}, \qquad \lambda^{(B)}_j= 
\dfrac{\lambda^{(G)}_j-\lambda^{(G)}_{(n-j) \bmod n}}{2}.
\end{eqnarray*}
For $r\in \{C, B\}$, set $\Lambda^{(r)}=
\text{diag}(\lambda^{(r)}_0$
$\lambda^{(r)}_1 \,
\cdots \, \lambda^{(r)}_{n-1})$, 
$C=Q_n \Lambda^{(C)} Q_n^T$ and
$B=Q_n \Lambda^{(B)} Q_n^T$.
Observe that ${\bm{\lambda}^{(G)}}=
{\bm{\lambda^{(C)}}} + {\bm{\lambda^{(B)}}}$,
where ${\bm{\lambda^{(C)}}}$ is symmetric and 
${\bm{\lambda^{(B)}}}$  is asymmetric.
Hence, $C \in {\mathcal C}_n$ and 
$B \in {\mathcal B}_n$. This proves \ref{directsum}.2).

\ref{directsum}.3) \, \, Let $C \in {\mathcal C}_n$ and 
$B \in {\mathcal B}_n$.  For $r\in \{C, B\}$, set $\Lambda^{(r)}=
\text{diag}(\lambda^{(r)}_0$
$\lambda^{(r)}_1 \,
\cdots \, \lambda^{(r)}_{n-1})$, 
$C=Q_n \Lambda^{(C)} Q_n^T$ and
$B=Q_n \Lambda^{(B)} Q_n^T$. Note that
${\bm{\lambda^{(C)}}}$ is symmetric and 
${\bm{\lambda^{(B)}}}$  is asymmetric. We have
$C + B= Q_n ( \Lambda^{(C)}+ 
\Lambda^{(B)}) Q_n^T
\in {\mathcal G}_n$.

\ref{directsum}.4) \, \, Choose arbitrarily 
$C\in{\mathcal C}_n$ and $B\in{\mathcal B}_n$. 
For $r\in \{C, B\}$, put $\Lambda^{(r)}=
\text{diag}(\lambda^{(r)}_0$
$\lambda^{(r)}_1 \,
\cdots \, \lambda^{(r)}_{n-1})$, 
$C=Q_n \Lambda^{(C)} Q_n^T$ and
$B=Q_n \Lambda^{(B)} Q_n^T$. Observe that 
${\bm{\lambda^{(C)}}}$ is symmetric and 
${\bm{\lambda^{(B)}}}$  is asymmetric.
%
In particular,
$\lambda^{(B)}_0=0$ and $\lambda^{(B)}_{n/2}=0$
when $n$ is even. Note that
$C^T B=
Q_n 
{\Lambda}^{(C)} \Lambda^{(B)} Q_n^T$,
where 
${\Lambda}^{(C)} {\Lambda}^{(B)}=$diag
$({\lambda^{(C)}_0} 
\lambda^{(B)}_0 \quad 
{\lambda^{(C)}_1} 
\lambda^{(B)}_1\, \,    
\cdots \, \,
{\lambda^{(C)}_{n-1}} 
\lambda^{(B)}_{n-1})$.
Thus, we obtain
\begin{eqnarray*}
\langle C ,B\rangle&=& \textrm{tr}(C^T B)=
\sum_{j=0}^{n-1}
{\lambda^{(C)}_j} \lambda^{(B)}_j= 
\sum_{j=1}^{\lfloor (n-1)/2 \rfloor}
(
{\lambda^{(C)}_j} \lambda^{(B)}_j+
{\lambda^{(C)}_{n-j}} \lambda^{(B)}_{n-j})
= \\ &=&\sum_{j=1}^{\lfloor (n-1)/2 \rfloor}
(
{\lambda^{(C)}_j} \lambda^{(B)}_j-
{\lambda^{(C)}_j} \lambda^{(B)}_j)=0, \nonumber
\end{eqnarray*}
that is \ref{directsum}.4). This ends the proof.  
\end{proof}
\begin{theorem}\label{fundamental12}
It is 
\begin{eqnarray}\label{fundamental10}
{\mathcal C}_n={\mathcal D}_n \oplus {\mathcal E}_n,
\end{eqnarray} 
where $\oplus$ is the orthogonal sum
with respect to the Frobenius product.
\end{theorem} 

\begin{proof}
Analogously as in Theorem \ref{directsum},
to get (\ref{fundamental10}) it is sufficient to prove 
the following properties:
\begin{description}
    \item[{\rm \ref{fundamental12}.1)}] ${\mathcal D}_n \cap
{\mathcal E}_n =\{O_n\}$,
where $O_n \in {\mathbb{R}}^{n \times n}$ is the matrix 
whose entries are equal to $0$;
    \item[{\rm \ref{fundamental12}.2)}]
for any $C\in{\mathcal C}_n$, there exist
$C_1 \in {\mathcal D}_n$ and 
$C_2 \in {\mathcal E}_n$ with 
$C=C_1 + C_2$;
   \item[{\rm \ref{fundamental12}.3)}] 
for every $C_1 \in {\mathcal D}_n$ and 
$C_2 \in {\mathcal E}_n$, we get that $C_1 + C_2 \in {\mathcal C}_n$.
   \item[{\rm \ref{fundamental12}.4)}] 
$\langle C_1,C_2 \rangle=0$ for each 
$C_1\in{\mathcal D}_n$ and 
$C_2\in{\mathcal E}_n$.
\end{description}

\ref{fundamental12}.1) \, \, 
Let $C\in{\mathcal D}_n \cap {\mathcal E}_n$. 
Then, $C=Q_n \Lambda Q_n^T$, where
$\Lambda=\text{diag}({\bm{\lambda}})$,
${\bm{\lambda}}$ is symmetric and such that $\lambda_0=0$
and $\lambda_{n/2}=0$ when $n$ is even, because 
$C\in{\mathcal D}_n$. Moreover, since $C\in{\mathcal E}_n$, 
we get that $\lambda_j =0$ for $j=1,\ldots, n-1$,
$j \neq n/2$ when  $n$ is  even, that is
${\bm{\lambda}}={\bm{0}}$, 
Thus, $\Lambda=O_n$ and hence $C=O_n$. This proves
\ref{fundamental12}.1).

\ref{fundamental12}.2) \, \, 
Let $C \in {\mathcal C}_n$, 
$\Lambda \in {\mathbb{R}}^{n \times n}$
be such that $C=Q_n \Lambda Q_n^T$, 
$\Lambda=$diag$({\bm{\lambda}})=\text{diag}(\lambda_0$
$\lambda_1 \,
\cdots \, \lambda_{n-1})$. For $j \in \{0$, $1, \ldots, n-1\}$,
set \begin{eqnarray*}
\lambda^{(1)}_j=\left\{
\begin{array}{ll}
\lambda_j & \text{if  }j \neq 0 \text{   and   }j \neq n/2,
\\ \\ 0 & \text{otherwise,  }
\end{array}\right. \qquad
\lambda^{(2)}_j=\left\{
\begin{array}{ll}
\lambda_j & \text{if  }j = 0 \text{   or   }j =n/2,
\\ \\ 0 & \text{otherwise. }
\end{array}\right. 
\end{eqnarray*}
For $r=1,2$, set $\Lambda_r=
\text{diag}(\lambda^{(r)}_0$
$\lambda^{(r)}_1 \,
\cdots \, \lambda^{(r)}_{n-1})$, and 
$C_r=Q_n \Lambda_r Q_n^T$.
Note that ${\bm{\lambda}}=
{\bm{\lambda^{(1)}}}+{\bm{\lambda^{(2)}}}$,
where ${\bm{\lambda^{(1)}}}$ is symmetric.
Hence, $C_1 \in {\mathcal D}_n$ and 
$C_2 \in {\mathcal E}_n$. This proves \ref{fundamental12}.2).

\ref{fundamental12}.3) \, \, Let $C_1 \in {\mathcal D}_n$ and 
$C_2 \in {\mathcal E}_n$.  
For $r=1,2$ there is
$\Lambda_r=$diag$({\bm{\lambda^{(r)}}})$=
$\text{diag}(\lambda^{(r)}_0$
 $\lambda^{(r)}_1 \,
\cdots \, \lambda^{(r)}_{n-1})$,
such that
$C_r=Q_n \Lambda_r Q_n^T$, 
${\bm{\lambda^{(1)}}}$ is symmetric, $\lambda^{(1)}_0=0$ and 
$\lambda^{(1)}_{n/2}=0$ when $n$ is even,
$\lambda^{(2)}_j=0$,
$j=1,\ldots, n-1$, $j \neq n/2$ when $n$ is even.
Thus, it is not difficult to check that
${\bm{\lambda^{(1)}}}+{\bm{\lambda^{(2)}}}$
is symmetric. So, 
$C_1 + C_2= Q_n$ diag$({\bm{\lambda^{(1)}}}+{\bm{\lambda^{(2)}}}) Q_n^T
\in {\mathcal C}_n$.

\ref{fundamental12}.4) \, \, Pick
$C_1\in{\mathcal D}_n$ and $C_2\in{\mathcal E}_n$.
Then for $r=1,2$ there exists 
$\Lambda_r=$diag$({\bm{\lambda^{(r)}}})$=
$\text{diag}(\lambda^{(r)}_0$
 $\lambda^{(r)}_1 \,
\cdots \, \lambda^{(r)}_{n-1})$,
such that
$C_r=Q_n \Lambda_r Q_n^T$, 
${\bm{\lambda^{(1)}}}$ is symmetric,
$\lambda^{(1)}_0=0$ and 
$\lambda^{(1)}_{n/2}=0$ when $n$ is even.
Note that
$C_1^T C_2=
Q_n 
{\Lambda_1} \Lambda_2 Q_n^T$,
where 
$$\Lambda_1 \Lambda_2= \textrm{diag}
({\lambda^{(1)}_0} \cdot
\lambda^{(2)}_0\, \, \, 
{\lambda^{(1)}_1} \cdot
\lambda^{(2)}_1\, \,    
\cdots \, \,
{\lambda^{(1)}_{n-1}} \cdot
\lambda^{(2)}_{n-1})=\textrm{diag}({\mathbf{0}})=O_n.$$
Therefore, $C_1^T C_2=O_n$, and thus we get
$\langle C_1,C_2\rangle
=$tr$(C_1^T C_2)=0$,
that is \ref{fundamental12}.4). This ends the proof.  
\end{proof}
Now we give a consequence of 
\ref{directsum} and \ref{fundamental12}.
\begin{corollary}
The following result holds:
\begin{eqnarray*}
{\mathcal G}_n={\mathcal B}_n \oplus {\mathcal D}_n 
\oplus {\mathcal E}_n.
\end{eqnarray*}
\end{corollary}
\begin{proposition}\label{productm}
Given two $\gamma$-matrices $G_1$, $G_2$, we get:

\begin{description}
\item[\rm {\ref{productm}.1)}] If $G_1$, $G_2\in 
{\mathcal C}_n$, then $G_1 \, G_2 \in {\mathcal C}_n$;
\item[\rm {\ref{productm}.2)}] If $G_1$, $G_2\in 
{\mathcal B}_n$, then $G_1 \, G_2 \in {\mathcal C}_n$;
\item[\rm {\ref{productm}.3)}] If $G_1 \in {\mathcal C}_n$ and
$G_2\in 
{\mathcal B}_n$, then $G_1 \, G_2
=G_2 \, G_1\in {\mathcal B}_n$.
\end{description}
\end{proposition}
\begin{proof} 
\ref{productm}.1) It follows immediately from the fact that 
${\mathcal G}_n$ is an algebra.

\ref{productm}.2) 
If $G_1=Q_n \Lambda^{({G_1})} 
Q_n^T $ and $G_2=Q_n \Lambda^{({G_2})} 
Q_n^T $, then
$G_1 \, G_2=Q_n 
\Lambda^{(G_1)}\, \Lambda^{(G_2)} Q_n^T$, 
where ${\Lambda}^{(G_1)} {\Lambda}^{(G_2)}=$diag
$({\lambda^{(G_1)}_0} 
\lambda^{(G_2)}_0 \quad 
{\lambda^{(G_1)}_1} 
\lambda^{(G_2)}_1\, \,    
\cdots \, \,
{\lambda^{(G_1)}_{n-1}} 
\lambda^{(G_1)}_{n-1})$.
Since the eigenvalues of $G_1$ and 
$G_2$ are asymmetric, we get that the eigenvalues of 
${G_1} \, G_2$ are symmetric. Hence,
${G_1} \, {G_2} \in {\mathcal C}_n$. 

\ref{productm}.3)  We first note that, since 
${\mathcal G}_n$ is an algebra, we get that $G_1\, G_2=
G_2 \, G_1$.
Since the eigenvalues of $G_1$ are symmetric and those of $G_2$ 
are asymmetric, arguing analogously
as in the proof of \ref{productm}.2)
it is possible to check that the eigenvalues of 
$G_1 \, G_2$ are asymmetric. Therefore,
${G_1} \, {G_2} \in {\mathcal B}_n$.   
\end{proof}
\section{Structural characterizations of $\gamma$-matrices}
In this section we show how ${\mathcal C}_n$ coincides
with the set of all real symmetric circulant matrices, and 
${\mathcal S}_n$ is a particular subclass of reverse circulant
matrices. 

We consider the set of families
\begin{eqnarray*}
{\mathcal L}_{n,k}&=& \{A \in {\mathbb{R}}^{n \times n}:
\text{   there  is   } \mathbf{a}=
(a_0 \, a_1 \, \ldots \, a_{n-1})^T \in {\mathbb{R}}^n \text{   with   }
a_{l,j}=a_{(j + kl) \bmod n} \}, \\
{\mathcal H}_{n,k}&=&
\{A \in {\mathbb{R}}^{n \times n}: \text{ there  is  a  symmetric }
 \mathbf{a}=
(a_0 \, a_1 \, \ldots \, a_{n-1})^T \in {\mathbb{R}}^n 
\\ & & \text{   with   }
a_{l,j}=a_{(j + kl) \bmod n} \},
\\ {\mathcal J}_{n,k}&=& 
\Bigl\{A \in {\mathbb{R}}^{n \times n}: \text{  there  is  a  symmetric }
 \mathbf{a}=
(a_0 \, a_1 \, \ldots \, a_{n-1})^T \in {\mathbb{R}}^n \text{   with   }\\ & &
\displaystyle{\sum_{t=0}^{n-1} a_t = 0},   
\displaystyle{\sum_{t=0}^{n-1} (-1)^t a_t
\, = \, 0} \text{  when  } n \text{  is  even,}
 \text{   and   }
a_{l,j}=a_{(j + kl) \bmod n} \Bigr\},
\end{eqnarray*}
where $k \in \{1,2,\ldots, n-1 \}$.

When $k=n-1$, ${\mathcal L}_{n,n-1}$ is the 
class of all \emph{real circulant matrices},
that is the family
of those matrices $C\in {\mathbb{R}}^{n \times n}$ such that 
every row, after the first, has 
the elements of the previous one shifted cyclically one place 
right (see, e.g., \cite{DAVIS}).

Given a vector $\mathbf{c} \in \mathbb{R}^n$,
 $\mathbf{c} = (c_0 \, c_1 \cdots 
c_{n-1})^T$, let us define
\begin{eqnarray*}\label{circ}
\text{circ}(\mathbf{c})=C=
\begin{pmatrix} 
c_0 & c_1 & c_2 &\ldots & c_{n-2} & c_{n-1} \\
c_{n-1} & c_0 & c_1 &\ldots & c_{n-3} & c_{n-2} \\
c_{n-2} & c_{n-1} & c_0 &\ddots & c_{n-4} & c_{n-3} \\
\vdots & \vdots & \ddots &\ddots & \ddots & \vdots \\
c_2 & c_3 & c_4 &\ddots & c_0 & c_1 \\
c_1 & c_2 & c_3 &\ldots & c_{n-1} & c_0 
\end{pmatrix},
\end{eqnarray*}
where $C \in {\mathcal L}_{n,n-1}$.

If $\mathrm{i}$ is the imaginary unit and 
$\omega_n=e^{\frac{2 \pi \mathrm{i}}{n}}$, then
the $n$-th roots of $1$ are 
\begin{eqnarray*}\label{rootsofunity}
\omega_n^j=
e^{\frac{2 \pi j \mathrm{i}}{n}}=
\cos \Bigl( \dfrac{2\pi j}{n}\Bigl) + 
\mathrm{i} \sin \Bigl( \dfrac{2\pi j}{n}\Bigl) ,
\qquad j=0, 1, \ldots, n-1.
\end{eqnarray*}
The \emph{Fourier matrix} of dimension $n \times n$ is defined by 
$F_n=(f^{(n)}_{k,l})_{k,l}$, where 
\begin{eqnarray*}\label{fourier}
f_{k,l}^{(n)}=\dfrac{1}{\sqrt{n}} \, \omega_n^{kl},
\qquad k,l=0,1,\ldots, n-1.
\end{eqnarray*}
Note that $F_n$ is symmetric, and $F_n^{-1}=F_n^*$
(see, e.g., \cite{DAVIS}).

Let ${\mathcal F}_n$ be the space of all real matrices
\emph{simultaneously diagonalizable} by $F_n$, that is 
\begin{eqnarray*}
{\mathcal F}_n= {\rm{sd  } } (F_n) = \{ F_n \Lambda F_n^*
\in {\mathbb{R}}^{n \times n}: 
\Lambda 
=\text{diag}(\bm{\lambda}), \, \bm{\lambda} \in \mathbb{C}^n
\}.
\end{eqnarray*}
It is not difficult to see that ${\mathcal F}_n$ 
is a commutative matrix algebra.
\begin{theorem} \label{characterizationcirculant}
\rm (\cite[Theorems 3.2.2 and 3.2.3]{DAVIS}) \em 
The following result holds:
\begin{eqnarray*}
{\mathcal F}_n={\mathcal L}_{n,n-1}.
\end{eqnarray*}
\end{theorem}
As a consequence of this theorem,
we get that
the $n$ eigenvectors of every circulant matrix
$C\in {\mathbb{R}}^{n \times n}$ are 
given by
\begin{eqnarray*}\label{tj}
{\mathbf{w}}^{(j)}=
(1 \, \, \omega_n^j \, \, 
\omega_n^{2j} \,  \, \cdots \, \, 
\omega_n^{(n-1)j})^T, 
\end{eqnarray*} and
the eigenvalues of a matrix $C =$circ$(\mathbf{c})\in  
{\mathcal F}_n$ 
are expressed by
$$\displaystyle{\lambda_j= 
{\mathbf{c}}^T {\mathbf{w}}^{(j)} =
\sum_{k=0}^{n-1} c_k \omega_n^{jk}},
\qquad j=0, 1, \ldots, n-1 .$$ 

Now we present some results about
symmetric circulant real matrices. 
Observe that, if $C=$circ$({\mathbf{c}})$, with $\mathbf{c} \in \mathbb{R}^n$,
then $C$ is symmetric if and only if ${\mathbf{c}}$ is symmetric.
Thus, the class of all real symmetric circulant matrices 
coincides with ${\mathcal H}_{n,n-1}$ and 
has dimension
$\lfloor \frac{n}{2} \rfloor +1$ over $\mathbb{R}$. 
%
\begin{theorem}\label{Theorem3.2} 
\rm (see, e.g., \cite[\S4]{CGV},
\cite[Lemma 3]{CINESI}) \em
Let $C \in {\mathcal H}_{n,n-1}$. Then, the set of all eigenvectors of $C$ can be
expressed as
$\{{\mathbf{q}}^{(0)}$, ${\mathbf{q}}^{(1)}$,
$\ldots$, ${\mathbf{q}}^{(n-1)}\}$, where 
${\mathbf{q}}^{(j)}$, $j=0$, $1, \ldots, n-1$, is
as in \rm (\ref{35}), (\ref{yjzj}) \em  and  \rm(\ref{45}). \em
%
\end{theorem}
Note that from Theorem \ref{Theorem3.2} it follows that
the set of all real symmetric circulant matrices is contained in ${\mathcal G}_n$. The nest result holds.
\begin{theorem}\label{Theorem3.1} \rm
(see, e.g., \cite[\S1.2]{BOSE}, \cite[\S4]{CGV}, 
\cite[Theorem 1]{TEE}) \em 
Let  $C=\text{circ}(\mathbf{c}) \in {\mathcal H}_{n,n-1}$.
Then, the eigenvalues $\lambda_j$ of $C$, 
$j=0$, $1, \ldots, \lfloor \frac{n}{2} \rfloor$, are given by  
\begin{eqnarray}\label{tee}
\lambda_j=
\mathbf{c}^T \mathbf{u}^{(j)}.
\end{eqnarray} 
Moreover, for $j=1$, $2, \ldots, \lfloor \frac{n-1}{2} \rfloor$ it is
\begin{eqnarray*}\label{symmetriceigenvaluese}
\lambda_j=\lambda_{n-j}.
\end{eqnarray*}
\end{theorem}  
From Theorem \ref{Theorem3.1} it follows that, if $C$ is
a real symmetric circulant matrix and 
$\bm{\lambda}^{(C)}$ is the set
of its eigenvalues, then $\bm{\lambda}^{(C)}$ is symmetric, thanks to 
(\ref{tee}). Hence,
\begin{eqnarray}\label{Hnn-1Sn}
{\mathcal H}_{n,n-1} \subset {\mathcal C}_n.
\end{eqnarray} 

Now we prove that ${\mathcal C}_n$ is contained in the class 
of all real symmetric circulant matrices ${\mathcal H}_{n,n-1}$.
First, we give the following
\begin{theorem}\label{cccirc}
Every matrix $C \in {\mathcal C}_n$ is circulant, that is 
\begin{eqnarray}\label{L}
{\mathcal C}_n \subset {\mathcal L}_{n,n-1}.
\end{eqnarray}
\end{theorem}
\begin{proof}
Let $C \in {\mathcal C}_n$, $C=(c_{k,l})_{k.l}$ and 
$\Lambda^{(C)}=\text{diag}(\lambda_0^{(C)}$ 
$\lambda_1^{(C)} \, \cdots \, \lambda_{n-1}^{(C)})$ be such that 
$\lambda_j^{(C)}=\lambda_{(n-j) \bmod n}^{(C)}$ for every
$j \in \{0$, $1, \ldots, n-1\}$, and 
$C=Q_n \Lambda^{(C)} Q_n^T$.
We have 
\begin{eqnarray*}\label{lam1}
c_{k,l} = \sum_{j=0}^{n-1}
q_{k,j}^{(n)}\, \lambda_j^{(C)}\,q^{(n)}_{l,j}.
\end{eqnarray*}
From this we get, if $n$ is even,
\begin{eqnarray}\label{susy00} c_{k,l}=
 \lambda_0^{(C)} 
q^{(n)}_{k,0} q^{(n)}_{l,0} + \lambda_{n/2}^{(C)} 
q^{(n)}_{k,n/2} q^{(n)}_{l, n/2} +
\sum_{j=1}^{n/2-1} 
\lambda_j^{(C)}(
q^{(n)}_{k,j}q^{(n)}_{l,j}+q^{(n)}_{k,n-j}q^{(n)}_{l,n-j}),
\end{eqnarray} 
and, if $n$ is odd,
\begin{eqnarray}\label{susy001}  c_{k,l}=
\lambda_0^{(C)} \, q^{(n)}_{k,0} \, q^{(n)}_{l,0}+
\sum_{j=1}^{ (n-1)/2 } 
\lambda_j^{(C)} (
q^{(n)}_{k,j}q^{(n)}_{l,j}+q^{(n)}_{k,n-j}q^{(n)}_{l,n-j}).
\end{eqnarray}
When $n$ is even, 
from 
(\ref{prestartingmatrix}) and (\ref{susy00}) we deduce
\begin{eqnarray*}
c_{k,l}&=& 
\dfrac{\lambda_0^{(C)}}{n} + (-1)^{k-l} \, 
\dfrac{\lambda_{n/2}^{(C)}}{n}+
\\&+&
\dfrac{2}{n} \sum_{j=1}^{n/2-1} \lambda_j^{(C)}
\cdot \Bigg(
    \cos\left(\dfrac{2\pi kj}{n}\right)\cdot
   \cos\left(\dfrac{2\pi lj}{n}\right) +
\sin\left(\dfrac{2\pi kj}{n}\right)\cdot
\sin\left(\dfrac{2\pi lj}{n}\right)\Bigg)= 
\\ &=& \dfrac{\lambda_0^{(C)}}{n} + (-1)^{k-l} \,\dfrac{\lambda_{n/2}^{(C)}}{n}+
\dfrac{2}{n}
\sum_{j=1}^{n/2-1} \lambda_j^{(C)}\cdot
    \cos\left(\dfrac{2\pi(k-l)j}{n}\right).
\end{eqnarray*}
Let $\mathbf{c}=(c_0 \, \, c_1 \, \cdots \, c_{n-1})^T$, where  
\begin{eqnarray*}
\displaystyle{c_t=\dfrac{\lambda_0^{(C)}}{n} + (-1)^t 
\,\dfrac{\lambda_{n/2}^{(C)}}{n}+\dfrac{2}{n}
\sum_{j=1}^{n/2-1} \lambda_j^{(C)} \cdot
\cos\left(\dfrac{2\,\pi\, t \, j}{n}\right)}, \quad 
t \in \{0, 1, \ldots, n-1\}. 
\end{eqnarray*}
Then we get $C=$circ$(\mathbf{c})$, since for any 
$k$, $l\in \{0$, $1, \ldots, n-1\}$ it is 
$c_{k,l}=c_{(k-l)\bmod n}$. 

When $n$ is odd, from 
(\ref{prestartingmatrix}) and (\ref{susy001}) we obtain
\begin{eqnarray*}
c_{k,l}&=& 
\dfrac{\lambda_0^{(C)}}{n} + 
\dfrac{2}{n} \sum_{j=1}^{(n-1)/2} \lambda_j^{(C)}
\cdot \Bigg(
    \cos\left(\dfrac{2\pi kj}{n}\right)\cdot
   \cos\left(\dfrac{2\pi lj}{n}\right) +\\ &+&
\sin\left(\dfrac{2\pi kj}{n}\right)\cdot
\sin\left(\dfrac{2\pi lj}{n}\right)\Bigg)= 
\\ &=& \dfrac{\lambda_0^{(C)}}{n} +
\dfrac{2}{n}
\sum_{j=1}^{(n-1)/2} \lambda_j^{(C)}\cdot
    \cos\left(\dfrac{2\pi(k-l)j}{n}\right).
\end{eqnarray*}

Let $\mathbf{c}=(c_0 \, c_1 \, \cdots \, c_{n-1})^T$, where  
\begin{eqnarray*}
\displaystyle{c_t=\dfrac{\lambda_0^{(C)}}{n} +\dfrac{2}{n}
\sum_{j=1}^{(n-1)/2} \lambda_j^{(C)} \cdot
\cos\left(\dfrac{2\pi \, t \, j}{n}\right)}, \quad 
t \in \{0, 1, \ldots, n-1\}. 
\end{eqnarray*}
Hence, $C=$circ$(\mathbf{c})$, because for each
$k$, $l\in \{0$, $1, \ldots, n-1\}$ it is
$c_{k,l}=c_{(k-l)\bmod n}$. 
Therefore,
${\mathcal C}_n \subset {\mathcal L}_{n,n-1}. $
\end{proof}
A consequence of Theorem \ref{cccirc} 
is the following
\begin{corollary}
The class ${\mathcal C}_n$ is the set of all real symmetric circulant matrices,
that is \begin{eqnarray}\label{HS}
{\mathcal C}_n={\mathcal H}_{n,n-1}.
\end{eqnarray}
\end{corollary}
\begin{proof} Since every matrix belonging to ${\mathcal C}_n$
is symmetric, we get that (\ref{HS}) is a consequence of 
(\ref{Hnn-1Sn}) and (\ref{L}).   
\end{proof}
If $k=1$, then ${\mathcal L}_{n,1}$ is the set of
all \emph{real reverse circulant} (or \emph{real anti-circulant}) 
\emph{matrices}, that is the 
class of all matrices $B\in {\mathbb{R}}^{n \times n}$ such that 
every row, after the first, has 
the elements of the previous one shifted cyclically one place left
(see, e.g., \cite{DAVIS}).
Given a vector 
 $\mathbf{b} = (b_0 \, b_1 \cdots 
b_{n-1})^T \in \mathbb{R}^n$, set
\begin{eqnarray*}\label{anticirc}
\text{rcirc}(\mathbf{b})=B=
\begin{pmatrix} 
b_0 & b_1 & b_2 &\ldots & b_{n-2} & b_{n-1} \\
b_1 & b_2 & b_3 &\ldots & b_{n-1} & b_0 \\
b_2 & b_3 & b_4 &\ldots & b_0 & b_1 \\
\vdots & \vdots & \vdots &\ldots & \vdots & \vdots \\
b_{n-2} & b_{n-1} & b_0 &\ldots & b_{n-4} & b_{n-3} \\
b_{n-1} & b_0 & b_1 &\ldots & b_{n-3} & b_{n-2}
\end{pmatrix},
\end{eqnarray*}
with $B \in {\mathcal L}_{n,1}$.

Observe that every
matrix $B \in  
 {\mathcal B}_{n,1}$ is symmetric,
and the set ${\mathcal L}_{n,1}$ 
is a linear 
space over $\mathbb{R}$, but not an algebra. 
Note that,
if $B_1$, $B_2 \in {\mathcal L}_{n,1}$, then 
$B_1 \, B_2$, $B_2 \, B_1 \in {\mathcal L}_{n, n-1}$ (see 
\cite[Theorem 5.1.2]{DAVIS}).

Now we give the next result.
\begin{theorem}\label{213}
The following inclusion holds:
$${\mathcal B}_n \subset  
{\mathcal L}_{n,1}.$$
\end{theorem}
\begin{proof}
Let $B\in {\mathcal B}_n$, $B=(b_{k,l})_{k.l}$ and 
$\Lambda^{(B)}=\text{diag}(\lambda_0^{(B)}$ 
$\lambda_1^{(B)} \, \cdots \, \lambda_{n-1}^{(B)})$ be such that 
$\lambda_j^{(B)}=-\lambda^{(B)}_{(n-j) \bmod n}$ for every
$j \in \{0$, $1, \ldots, n-1\}$, and $B=Q_n \Lambda^{(B)} Q_n^T$.
We have 
\begin{eqnarray}\label{lam00}
b_{k,l} = \sum_{j=0}^{n-1}
q^{(n)}_{k,j}\, \lambda_j^{(B)} \,q^{(n)}_{l,j}.
\end{eqnarray}
Observe that $\lambda_0^{(B)}=0$ and 
$\lambda_{n/2}^{(B)}=0$, if
$n$ is even. From this and (\ref{lam00}) we get
\begin{eqnarray}\label{susy2}
b_{k,l}= \sum_{j=1}^{\lfloor (n-1)/2 \rfloor} 
\lambda_j^{(B)} \cdot \big(
q^{(n)}_{k,j}q^{(n)}_{l,j}-q^{(n)}_{k,n-j}q^{(n)}_{l,n-j}\big),
\end{eqnarray}
both when $n$ is even and when $n$ is odd. From
(\ref{prestartingmatrix}) and (\ref{susy2}) we deduce
\begin{eqnarray*}
b_{k,l}&=& 
\dfrac{2}{n} \sum_{j=1}^{\lfloor (n-1)/2 \rfloor} \lambda_j^{(B)}
 \Bigg(
    \cos\left(\dfrac{2\pi kj}{n}\right)
   \cos\left(\dfrac{2\pi lj}{n}\right)-
\sin\left(\dfrac{2\pi kj}{n}\right)
\sin\left(\dfrac{2\pi lj}{n}\right)\Bigg)= 
\\ &=& \dfrac{2}{n}
\sum_{j=1}^{\lfloor (n-1)/2 \rfloor} \lambda_j^{(B)}
    \cos\left(\dfrac{2\pi(k+l)j}{n}\right).
\end{eqnarray*}
Let $\mathbf{b}=(b_0 \, b_1 \, \cdots \, b_{n-1})^T$, where  
\begin{eqnarray}\label{ct}
\displaystyle{b_t=\dfrac{2}{n}
\sum_{j=1}^{\lfloor (n-1)/2 \rfloor} \lambda_j^{(B)} \cdot
\cos\left(\dfrac{2\pi tj}{n}\right)}, \quad 
t \in \{0, 1, \ldots, n-1\}. 
\end{eqnarray}
Thus, $B=$circ$(\mathbf{b})$, because for each
$k$, $l\in \{0$, $1, \ldots, n-1\}$ we have
$b_{k,l}=b_{(k-l)\bmod n}$. 
For any 
$k$, $l\in \{0$, $1, \ldots, n-1\}$ it is 
$b_{k,l}=b_{(k+l)\bmod n}$. Hence,
${\mathcal B}_n \subset {\mathcal L}_{n,1}$.  
\end{proof}
\begin{theorem}\label{thm:A*n struc}
One has
\begin{eqnarray*}\label{bypass}
{\mathcal B}_n\subset{\mathcal H}_{n,1}.
\end{eqnarray*}
\end{theorem}
\begin{proof}
We recall that \begin{eqnarray*}
{\mathcal H}_{n,1}&=& 
\Bigl\{B \in {\mathbb{R}}^{n \times n}: \text{  there  is  a  symmetric }
 \mathbf{b}=
(b_0 \, b_1 \, \ldots \, b_{n-1})^T \in {\mathbb{R}}^n 
\\ & & \text{   with   }
b_{k,j}=b_{(j + k) \bmod n} \Bigr\}.
\end{eqnarray*}
By Theorem \ref{213}, we get ${\mathcal B}_n \subset {\mathcal L}_{n,1}$.
Now we prove the symmetry of ${\mathbf{b}}$.

Let $B \in {\mathcal B}_n$ be such that there exists
$\Lambda^{(B)} \in {\mathbb{R}}^{n \times n}$,
$\Lambda^{(B)}=\text{diag}(\lambda^{(B)}_0$
$\lambda^{(B)}_1 \,
\cdots \, \lambda^{(B)}_{n-1})$, such that   
$C= Q_n \Lambda^{(B)} Q_n^T$ and 
$\lambda_j^{(B)}=
-\lambda^{(B)}_{(n-j)\bmod n}$
for all $j \in \{0$, $1, \ldots, n-1 \}$. By Theorem \ref{213},
$b_{k,j}=b_{(j + k) \bmod n}$. Moreover,
by arguing as in Theorem \ref{213}, we get (\ref{ct}), and hence
\begin{eqnarray*}\label{at}
b_t&=&\dfrac{2}{n}
\sum_{j=1}^{\lfloor (n-1)/2 \rfloor} \lambda_j^{(B)} \cdot
\cos\left(\dfrac{2\pi tj}{n}\right)= \dfrac{2}{n}
\sum_{j=1}^{\lfloor (n-1)/2 \rfloor} \lambda_j^{(B)} \cdot
\cos\left( 2 \pi j - \dfrac{2\pi tj}{n}\right)=\\
&=& \dfrac{2}{n} \sum_{j=1}^{\lfloor (n-1)/2 \rfloor}
\lambda_j^{(B)} \cdot
\cos\left(\dfrac{2\pi (n-t)j}{n}\right)=b_{n-t}
\end{eqnarray*} 
for any $t \in \{0$, $1, \ldots, n-1 \}$. 
Thus, $\mathbf{b}$ is symmetric.  
\end{proof}
\begin{theorem}\label{Theorem3.126} 
Let  $B=$\rm rcirc\em$(\mathbf{b}) \in {\mathcal B}_n$.
Then, the eigenvalues $\lambda_j^{(B)}$ of $B$, 
$j=0, 1, \ldots, \lfloor \frac{n}{2} \rfloor$, can be expressed as
\begin{eqnarray}\label{teeminus}
\lambda_j^{(B)}=
\mathbf{b}^T \mathbf{u}^{(j)}.
\end{eqnarray}
Moreover, for $j=1$, $2, \ldots \lfloor \frac{n-1}{2} \rfloor$, we get
$$\lambda^{(B)}_{n-j}=-\lambda^{(B)}_j.
$$Furthermore, it is 
$\lambda_0^{(B)}=0$, and $\lambda_{n/2}^{(B)}=0$ if $n$ is even.
\end{theorem} 
\begin{proof}
Since ${\mathbf{u}}^{(j)}$, $j=0,1, \ldots \lfloor n/2 \rfloor$,
is an eigenvector of $B$, then $$B \, {\mathbf{u}}^{(j)} =
\lambda_j^{(B)} \, {\mathbf{u}}^{(j)},$$ that is every component of the vector
$B \,  {\mathbf{u}}^{(j)}$ is equal to the respective component 
of the vector $\lambda_j^{(B)}\, {\mathbf{u}}^{(j)}$. In particular, if we consider 
the first component, we obtain (\ref{teeminus}).
The last part of the assertion is a consequence of the asymmetry 
of the vector ${\bm{\lambda}}^{(B)}$.    
\end{proof}
For the general computation of the eigenvalues 
of reverse circulant matrices, see, e.g., 
\cite[\S1.3 and Theorem 1.4.1]{BOSE},
\cite[Lemma 4.1]{SC1}.
\begin{theorem}\label{thm:A*n struc}
The following result holds:
$${\mathcal B}_n={\mathcal J}_{n,1}.$$
\end{theorem}
\begin{proof} 
First of all, 
we recall that 
\begin{eqnarray*}
{\mathcal J}_{n,1}&=& 
\Bigl\{B \in {\mathbb{R}}^{n \times n}: \text{  there  is  a  symmetric }
 \mathbf{b}=
(b_0 \, b_1 \, \ldots \, b_{n-1})^T \in {\mathbb{R}}^n \text{   with   }\\ &&
\displaystyle{\sum_{t=0}^{n-1} b_t = 0} , 
\displaystyle{\sum_{t=0}^{n-1} (-1)^t b_t
\, = \, 0} \text{ when  }n \text{  is  even,  } \text{   and   }
b_{k,j}=b_{(j + k)\bmod n}\Bigr\}.
\end{eqnarray*}
We begin with proving that ${\mathcal B}_n \subset {\mathcal J}_{n,1}$.

Let $B \in {\mathcal B}_n$. In Theorem \ref{thm:A*n struc}
we proved that $B \in {\mathcal H}_{n,1}$, that is
$\mathbf{b}$ is symmetric and $b_{k,j}=b_{(j + k)\bmod n}$.

Now we prove that 
\begin{eqnarray}\label{zerosum11}
\displaystyle{\sum_{t=0}^{n-1} b_t = 0}.
\end{eqnarray}
Since $B \in {\mathcal B}_n$, the vector
$${\mathbf{u}}^{(0)} =  \, \, \Bigl( 1  \, \, 
 1 \, \, \cdots \, \,  1 \Bigr)^T$$
is an eigenvector for the eigenvalue 
$\lambda_0^{(B)}=0$. Hence, the formula (\ref{zerosum11})
is a consequence of (\ref{teeminus}). \\

Again by (\ref{teeminus}), we get $$\displaystyle{\sum_{t=0}^{n-1} (-1)^t b_t
\, = \, 0},$$
since the vector $$
{\mathbf{u}}^{(n/2)} = 
\Bigl( 1  \, -1 \, \, \, \, \, \, 1 \, -1 \, \cdots \, -1\Bigr)^T$$
is an eigenvector for the eigenvalue 
$\lambda^{(B)}_{n/2}=0$ if $n$ is even. 
Thus, ${\mathcal B}_n \subset {\mathcal J}_{n,1}$.

Now observe that ${\mathcal J}_{n,1}$ is a linear space 
of dimension $\lfloor (n-1)/2 \rfloor$. Thus, by Proposition \ref{bn},
${\mathcal B}_n $ and ${\mathcal J}_{n,1}$ have the same dimension.
So, ${\mathcal B}_n = {\mathcal J}_{n,1}$. This ends the proof.
  \end{proof}
\begin{theorem}\label{thm:A*n struc}
The following result holds:
$${\mathcal D}_n={\mathcal J}_{n,n-1}.$$
\end{theorem}
\begin{proof}
We recall that \begin{eqnarray*}
{\mathcal J}_{n,n-1}&=& 
\Bigl\{C \in {\mathbb{R}}^{n \times n}: \text{  there  is  a  symmetric }
 \mathbf{c}=
(c_0 \, c_1 \, \ldots \, c_{n-1})^T \in {\mathbb{R}}^n \text{   with   }\\ &&
\displaystyle{\sum_{t=0}^{n-1} c_t = 0}, 
\displaystyle{\sum_{t=0}^{n-1} (-1)^t c_t
\, = \, 0} \text{  when   }n \text{  is  even, and  }
c_{k,j}=c_{(j - k) \bmod n}\Bigr\}.
\end{eqnarray*}
We first prove that ${\mathcal D}_n \subset {\mathcal J}_{n,n-1}$.
From Theorem \ref{fundamental12} and (\ref{Hnn-1Sn}) we deduce that 
${\mathcal D}_n \subset {\mathcal C}_n={\mathcal H}_{n,n-1}$.
Therefore, if $C=$circ$({\mathbf{c}}) \in {\mathcal D}_n$,
then ${\mathbf{c}}$ is symmetric.

Now we prove that 
\begin{eqnarray}\label{zerosum1}
\displaystyle{\sum_{t=0}^{n-1} c_t = 0}.
\end{eqnarray}
Since $C \in {\mathcal C}_n$, the vector
$${\mathbf{u}}^{(0)} = 
\Bigl( 1  \, \, 
 1 \, \, \cdots \, \,  1 \Bigr)^T$$
is an eigenvector for the eigenvalue 
$\lambda_0=0$. Hence, the formula (\ref{zerosum1})
is a consequence of (\ref{tee}). \\

Again by (\ref{tee}), we get $$\displaystyle{\sum_{t=0}^{n-1} (-1)^t c_t
\, = \, 0},$$ 
since the vector $$
{\mathbf{u}}^{(n/2)} = 
\Bigl( 1  \, -1 \, \, \, \, \, 1 \, -1 \, \cdots \, -1\Bigr)^T$$
is an eigenvector for the eigenvalue 
$\lambda_{n/2}=0$ if $n$ is even. 
Thus, ${\mathcal D}_n \subset {\mathcal J}_{n,n-1}$.

Now observe that ${\mathcal J}_{n,n-1}$ is a linear space 
of dimension $\lfloor (n-1)/2 \rfloor$. Thus, by Proposition \ref{bn},
${\mathcal D}_n $ and ${\mathcal J}_{n,n-1}$ have the same dimension.
So, ${\mathcal D}_n = {\mathcal J}_{n,n-1}$. This completes the proof.  
\end{proof}
\begin{theorem}\label{Encharacterization}
The next result holds: 
\begin{eqnarray*}\label{chess}
{\mathcal E}_n ={\mathcal P}_n= {\mathcal L}_{n,n-1} \cap {\mathcal L}_{n,1},
\end{eqnarray*}
where
\begin{eqnarray*}
{\mathcal P}_n= \left\{
\begin{array}{ll}
\left\{ C \in{\mathbb{R}}^{n \times n:} \text{ there are } k_1, k_2
 \text{ with } c_{i,j}=\left\{
\begin{array}{ll}
 k_1 & \text{if  } i+j \text{  is  even}
\\ k_2 & \text{if  } i+j \text{  is  odd}
\end{array}\right. \, \, \right\} & \text{if  }n \text{  is  even}, \\
\\ \{ C \in{\mathbb{R}}^{n \times n:} \text{ there is } k
 \text{ with } c_{i,j}= k \text{ for all } i,j=0,1,\ldots, n-1 \}
 & \text{if  }n \text{  is  odd}.
\end{array} \right.
\end{eqnarray*}
\end{theorem}
\begin{proof}
We first claim that ${\mathcal E}_n = {\mathcal P}_n$.

We begin with the inclusion ${\mathcal E}_n \subset {\mathcal P}_n$.
Let $C \in {\mathcal E}_n$, $C=(c_{k,l})_{k.l}$ and 
$\Lambda=\text{diag}(\lambda_0$ 
$\lambda_1 \, \cdots \, \lambda_{n-1})$ be such that 
$\lambda_0=0$ for every
$j \in \{1, 2, \ldots, n-1\}$ except $n/2$ when $n$ is even,
and $C=Q_n \Lambda Q_n^T$.
We have 
\begin{eqnarray*}\label{lam1}
c_{k,l} = \sum_{j=0}^{n-1}
q_{k,j}^{(n)}\, \lambda_j \,q^{(n)}_{l,j}.
\end{eqnarray*} 
From this we get
\begin{eqnarray}\label{susy} c_{k,l}=
\left\{ \begin{array}{ll} \lambda_0
\, q^{(n)}_{k,0} \, q^{(n)}_{l,0} + \lambda_{n/2} 
\, q^{(n)}_{k,n/2} \, q^{(n)}_{l, n/2}  
& \text{if  } n \text{   is even,} \\ \\ \lambda_0 \, q^{(n)}_{k,0} \, q^{(n)}_{l,0}
& \text{if  } n \text{   is odd.}
\end{array}\right.
\end{eqnarray}
When $n$ is even, 
from 
(\ref{prestartingmatrix}) and (\ref{susy}) we deduce
\begin{eqnarray*}
c_{k,l}&=& 
\dfrac{\lambda_0}{n} + (-1)^{k-l} \, \dfrac{\lambda_{n/2}}{n}.
\end{eqnarray*}
Note that 
\begin{eqnarray*}
c_{k,l}=\left\{ \begin{array}{ll}
\dfrac{\lambda_0}{n}+\dfrac{\lambda_{n/2}}{n} & 
\text{if   } k-l \text{  is even,}\\ \\
\dfrac{\lambda_0}{n}-\dfrac{\lambda_{n/2}}{n} & 
\text{if   } k-l \text{  is odd,}
\end{array}\right.
\end{eqnarray*}
and thus $C \in {\mathcal P}_n$.

When $n$ is odd, from 
(\ref{prestartingmatrix}) and (\ref{susy}) we obtain
\begin{eqnarray*}
c_{k,l}&=& 
\dfrac{\lambda_0}{n}
\end{eqnarray*}
for $k$, $l \in \{0, 1, \ldots, n-1\}$.
Hence, $C \in {\mathcal P}_n$.

Now we prove that ${\mathcal P}_n \subset {\mathcal E}_n$.
First of all, note that ${\mathcal P}_n \subset {\mathcal C}_n$.
Let $C \in {\mathcal P}_n$. If $n$ is even, then rank$(C) \leq 2$, and
hence $C$ has at least $n-2$ eigenvalues equal to $0$. 
By contradiction, suppose that at least one of the eigenvalues different
from $\lambda_0$ and $\lambda_{n/2}$, say $\lambda_j$, 
is different from $0$. Thus, 
$\lambda_0=0$ or $\lambda_{n/2}=0$. If $\lambda_0=0$, then 
\begin{eqnarray*}\label{firstrow}
\displaystyle{\sum_{t=0}^{n-1} c_t = 0},
\end{eqnarray*}
and hence $k_1=-k_2$. Therefore, rank$(C)\leq 1$, and thus there
is at most one non-zero eigenvalue. This implies that $\lambda_{n/2}=0$,
and hence 
\begin{eqnarray}\label{leibnitz}
\displaystyle{\sum_{t=0}^{n-1} (-1)^t c_t
\, = \, 0}.
\end{eqnarray}
From (\ref{leibnitz}) it follows that $k_1=k_2=0$. Thus we deduce that
$C=O_n$, which obviously implies that $\lambda_j=0$. This is absurd,
and hence $\lambda_0 \neq 0$.

When $\lambda_{n/2}=0$, then (\ref{leibnitz}) holds, and hence
$k_1=k_2$. Thus, rank$(C)\leq 1$, which implies that $\lambda_0=0$,
because we know that $\lambda_j \neq 0$. This yields a contradiction.
Thus, ${\mathcal P}_n \subset {\mathcal E}_n$, at least when 
$n$ is even.

Now we suppose that $n $ is odd. If $C\in {\mathcal P}_n$, then 
rank$(C)\leq 1$. This implies that $C$ has at most a non-zero eigenvalue. We claim 
that $\lambda_j=0$ for all $j \in \{1,2,\ldots, n-1\}$.
By contradiction, suppose that there exists $q \in \{1,2,\ldots, n-1\}$
such that
$\lambda_q \neq 0$. Hence, $\lambda_0=0$, and thus  
\begin{eqnarray*}
\displaystyle{0=\sum_{t=0}^{n-1} c_t = n \,k}.
\end{eqnarray*}
This implies that
$C=O_n$. Hence, $\lambda_q=0$, which is absurd.
Therefore, ${\mathcal P}_n \subset {\mathcal E}_n$
even when $n$ is odd.

Now we claim that 
${\mathcal P}_n= {\mathcal L}_{n,n-1} \cap {\mathcal L}_{n,1}$.

Observe that, if $C \in {\mathcal P}_n$, then $C$ is
both circulant and reverse circulant, and hence 
${\mathcal P}_n\subset {\mathcal L}_{n,n-1} \cap {\mathcal L}_{n,1}$.
Now we claim that 
${\mathcal L}_{n,n-1} \cap {\mathcal L}_{n,1} \subset {\mathcal P}_n$.
Let $C \in {\mathcal L}_{n,n-1} \cap {\mathcal L}_{n,1}$. If
$\mathbf{c}=(c_0 \, c_1 \ldots c_{n-1})^T$ is the first row of $C$, 
and $C=$circ$(\mathbf{c})=$rcirc$(\mathbf{c})$, then we get
\begin{eqnarray*}
c_{1,j}= c_{(j-1) \bmod n}=c_{(j+1) \bmod n}, \quad 
j \in \{0, 1, \ldots, n-1 \}.
\end{eqnarray*}
If $n$ is even, then $c_{2j}=c_0$ 
and $c_{2j+1}=c_1$ for $j \in \{0, 1, \ldots, n/2 - 1 \}$,
while when $n$ is odd we have $c_j=c_0$ for $j \in \{0, 1, \ldots, n-1 \}$,
getting the claim.  
\end{proof}
\begin{theorem}
The following result holds:
$${\mathcal H}_{n,1} = {\mathcal B}_n \oplus {\mathcal E}_n.$$
\end{theorem}
\begin{proof}
By Theorem \ref{thm:A*n struc}, we know that
\begin{eqnarray*}
{\mathcal B}_n={\mathcal J}_{n,1}&=& 
\Bigl\{C \in {\mathbb{R}}^{n \times n}: \text{  there  is  a  symmetric }
 \mathbf{c}=
(c_0 \, c_1 \, \ldots \, c_{n-1})^T \in {\mathbb{R}}^n \text{   with   }\\ &&
\displaystyle{\sum_{t=0}^{n-1} c_t = 0} ,  
\displaystyle{\sum_{t=0}^{n-1} (-1)^t c_t
\, = \, 0 } \text{ if   }n \text{  is  even,  } 
\text{   and   }
c_{i,j}=c_{(j + i) \bmod n}\Bigr\}.
\end{eqnarray*}
Moreover, by Theorem \ref{Encharacterization} we have
\begin{eqnarray*}
{\mathcal E}_n&=& 
\Bigl\{C \in {\mathbb{R}}^{n \times n}: \text{  there  is  }
 \mathbf{c}=
(c_0 \, c_1 \, \ldots \, c_{n-1})^T \in {\mathbb{R}}^n \text{  such that:   }
\text{there is } k \in\mathbb{R} \text{   with   } \\ &&
c_t=k, \, \, t=0,1,\ldots, n-1 \text{  if  }n \text{  is  even,} \text{ and  }
k_1,k_2 \in\mathbb{R} \text{   with   }
c_t=k_{((-1)^t+3)/2} \\ && \text{if   }  n \text{  is odd,  } \, \, t=0,1,\ldots, n-1,
\, c_{i,j}=c_{(j + i) \bmod n} \Bigr\}.
\end{eqnarray*}
First, we show that  
\begin{eqnarray}
\label{firstinclusio}
{\mathcal H}_{n,1} 
\supset {\mathcal B}_n \oplus {\mathcal E}_n.
\end{eqnarray}
Let $C=$rcirc$(\mathbf{c}) \in {\mathcal B}_n \oplus {\mathcal E}_n$. There are 
$C=$rcirc$(\mathbf{c}^{(r)})$, with 
$C^{(1)} \in {\mathcal B}_n$, $C^{(2)} \in {\mathcal E}_n$.
$r=1,2$, $C=C^{(1)}+C^{(2)}$, $
\mathbf{c}= \mathbf{c}^{(1)}+ \mathbf{c}^{(2)} $.
Note that $\mathbf{c}$ is symmetric, because both
$ \mathbf{c}^{(1)}$ and $\mathbf{c}^{(2)}$ are.
Thus, (\ref{firstinclusio}) is proved.

We now prove the converse inclusion. Let 
$C=$rcirc$(\mathbf{c}) \in {\mathcal H}_{n,1}$. Then,
$\mathbf{c}$ is symmetric. Suppose that $n$ is odd and
$$\sum_{t=0}^{n-1} c_t=\tau.$$
Let $\mathbf{c}^{(2)}=(\tau/n \, \, \tau/n \ldots \tau/n)^T$, 
$\mathbf{c}^{(1)}= \mathbf{c} - \mathbf{c}^{(2)}$, and 
$C=$rcirc$(\mathbf{c}^{(r)})$, $r=1,2$.
Then, $C=C^{(1)}+C^{(2)}$. Note that $C^{(2)}\in {\mathcal E}_n$.
Moreover, $\mathbf{c}^{(1)}$ is symmetric,
and $$\sum_{t=0}^{n-1} c^{(1)}_t=
\sum_{t=0}^{n-1}(c_t- \tau/n)=0.$$
Therefore, $C^{(1)}\in {\mathcal B}_n$. 

Now assume that $n$ is even. We get
\begin{eqnarray}\label{alpha}
\sum_{t=0}^{n-1} c_t=\tau= \dfrac{n}{2} \, (k_1 + k_2)
\end{eqnarray} and
\begin{eqnarray}\label{alphabis}
\sum_{t=0}^{n-1} (-1)^t \, c_t=\gamma_*= \dfrac{n}{2} \, (k_1 - k_2).
\end{eqnarray}
Then, 
\begin{eqnarray}\label{beta}
k_1=(\tau+\gamma_*)/n, \qquad k_2=(\tau-\gamma_*)/n.
\end{eqnarray}
Let $\mathbf{c}^{(2)}=(k_1 \, k_2 \ldots k_1 \, k_2)^T$,
$\mathbf{c}^{(1)}= \mathbf{c} - \mathbf{c}^{(2)}$, and 
$C=$rcirc$(\mathbf{c}^{(r)})$, $r=1,2$.
Then, $C=C^{(1)}+C^{(2)}$. Note that $C^{(2)}\in {\mathcal E}_n$.
Moreover, $\mathbf{c}^{(1)}$ is symmetric, and from 
(\ref{alpha}), (\ref{beta}) we obtain
$$\sum_{t=0}^{n-1} c^{(1)}_t=
\sum_{t=0}^{n-1}c_t - \dfrac{n}{2} \, (k_1 + k_2)=\tau-\tau=0.$$ 
Moreover, from 
(\ref{alphabis}), (\ref{beta}) we deduce
$$\sum_{t=0}^{n-1} (-1)^t \, c^{(1)}_t=
\sum_{t=0}^{n-1}  (-1)^t \,c_t - \dfrac{n}{2} \, (k_1 - k_2)=
\gamma_*-\gamma_*=0.$$ Thus, $C^{(1)}\in {\mathcal B}_n$. 

Moreover observe that, by Theorem \ref{fundamental12}, 
${\mathcal E}_n \subset {\mathcal C}_n$, and thanks to
Theorem \ref{directsum}, ${\mathcal C}_n$ and ${\mathcal B}_n$
are orthogonal. This implies that ${\mathcal E}_n$ and ${\mathcal B}_n$
are orthogonal. 
This ends the proof.  
\end{proof}
\section{Multiplication between a $\gamma$-matrix and a real vector} 
In this section we deal with the problem of computing the product
\begin{eqnarray}\label{primordiale}
\mathbf{y}= G \mathbf{x},
\end{eqnarray}
where $G \in {\mathcal G}_n$. 
Since $G$ is a $\gamma$-matrix, there is a diagonal matrix $\Lambda^{(G)} \in 
{\mathbb{R}}^{n \times n}$ with 
$G=Q_n \Lambda^{(G)} Q_n^T$, where $Q_n$ 
is as in (\ref{prestartingmatrix}). To compute 
\begin{eqnarray*}\label{e}
\mathbf{y}= Q_n \Lambda^{(G)} Q_n^T \mathbf{x}, 
\end{eqnarray*}
we proceed by doing the next operations
in the following order:
\begin{eqnarray}\label{2.2}
\mathbf{z}= Q_n^T \mathbf{x};
\end{eqnarray}
\begin{eqnarray}\label{2.3}
\mathbf{t}= \Lambda^{(G)} \mathbf{z};
\end{eqnarray}
\begin{eqnarray}\label{2.4}
\mathbf{y}= Q_n \mathbf{t}.
\end{eqnarray}
To compute the product in (\ref{2.2}), we define 
a new fast technique, which we will call
\emph{inverse discrete sine-cosine transform}
(IDSCT), while to do the operation in (\ref{2.4}), we define
a new technique, which will be called
\emph{discrete sine-cosine transform}
(DSCT). 

From now on, we assume that the involved vectors 
$\mathbf{x}$ belong to $\mathbb{R}^n$,
where $n=2^r$, $r \geq 2$, and we put $m
={n}/{2}$, 
$\nu=
{n}/{4}$.
\subsection{The ${\rm IDSCT}$ technique}
In this subsection
we present the technique to 
compute (\ref{2.2}).
Given $\mathbf{x} \in {\mathbb{R}}^n$,
we denote by 
\begin{eqnarray}\label{IDSCT}
\textrm{IDSCT}(\mathbf{x})=Q_n^T \, \mathbf{x}=\left(
\begin{array}{c}
\alpha_0 \, {\textrm{C}}_0(\mathbf{x}) \\
\alpha_1 \, {\textrm{C}}_1(\mathbf{x})  \\ \vdots \\
\alpha_m \, {\textrm{C}}_m(\mathbf{x}) \\
\alpha_{m+1} \, {\textrm{S}}_{m-1}(\mathbf{x})\\ \vdots \\
\alpha_{n-1} \, {\textrm{S}}_1(\mathbf{x})
\end{array}\right) ,
\end{eqnarray}
where $\alpha_j$, $j=0$, $1, \ldots, n-1$, are as in 
(\ref{alphak}), 
\begin{eqnarray}\label{COS}
{\textrm{C}}_j(\mathbf{x})= \mathbf{x}^T {\mathbf{u}^{(j)}}, \quad
j=0, 1, \ldots, m,
\end{eqnarray}
\begin{eqnarray}\label{SIN}
{\textrm{S}}_j(\mathbf{x})= \mathbf{x}^T {\mathbf{v}^{(j)}},
\quad
j=1, 2, \ldots, m-1.
\end{eqnarray}

 For 
$\mathbf{x} = 
(x_0 \, x_1 \, \cdots x_{n-1} )^T \in \mathbb{R}^{n}$, 
let $\eta$ (resp., $\zeta$): $\mathbb{R}^n \to \mathbb{R}^{m}$ be
the function which associates to any vector $\mathbf{x}
\in \mathbb{R}^{n}$ 
the vector consisting of all even (resp., odd)
components of $\mathbf{x}$, that is
\begin{eqnarray*}\label{evenodd000} \nonumber 
\eta ( {\mathbf{x}})
&=& \eta \, \mathbf{x}= (
x^{\prime}_0 \, x^{\prime}_1 \, \ldots \, 
x^{\prime}_{m-1}),
\text{  where  } 
x^{\prime}_p =x_{2p},
\, \, p=0, 1, \ldots, m-1; \\
\zeta ( {\mathbf{x}})&=&\zeta \, \mathbf{x}= (
x^{\prime \prime}_0 \, 
x^{\prime \prime}_1 \, \ldots \, 
x^{\prime \prime}_{m-1}), 
\text{  where  }
x^{\prime \prime}_p=x_{2p+1},
\, \, p=0, 1, \ldots, m-1.
\end{eqnarray*} 
\begin{proposition}\label{coseno} \rm (see, e.g., \cite{gpcdm}) \em
For $k=0, 1, \ldots, m$, we have
\begin{eqnarray}\label{zkcos}
{\rm{C}}_k(\mathbf{x})={\rm{C}}_k(\eta \, \mathbf{x})
+ \cos \Bigl( \dfrac{\pi \, k}{m} \Bigr) \, 
{\rm{C}}_k  (\zeta \, \mathbf{x}) - \sin \Bigl( \dfrac{\pi \, k}{m} \Bigr)
\, {\rm{S}}_k (\zeta\,\mathbf{x}).
\end{eqnarray}
Moreover, for $k=1,2, \ldots, \nu$, it is
\begin{eqnarray}\label{zksin}
{\rm S}_k(\mathbf{x})={\rm{S}}_k(\eta \, \mathbf{x})
+ \cos \Bigl( \dfrac{\pi \, k}{m} \Bigr) \, 
{\rm{S}}_k(\zeta \, \mathbf{x}) + \sin \Bigl( \dfrac{\pi \, k}{m} \Bigr)
\, {\rm{C}}_k(\zeta \,\mathbf{x}).
\end{eqnarray}
\end{proposition}
\begin{proof}
We first prove (\ref{zkcos}). By
using (\ref{COS}),
for $k =0, 1,\ldots, m$ we have
\begin{eqnarray*}
{\textrm{C}}_k(\mathbf{x})&=&\mathbf{x}^T {\mathbf{u}^{(k)}}=
\nonumber
\sum_{j=0}^{n-1} x_j \cos \Bigl( \dfrac{2 \pi k j }{n} \Bigr)=
\\ &=&  \sum_{j \, \text{even}} x_j \nonumber
\cos \Bigl( \dfrac{2 \pi k j }{n} \Bigr) +  \sum_{j \, 
\text{odd}} x_j \cos \Bigl( \dfrac{2 \pi  k j}{n} \Bigr)=
\\ &=& \sum_{p=0}^{m-1} x_{2p} 
\cos \Bigl( \dfrac{4 \pi  k p }{n} \Bigr) +
\sum_{p=0}^{m-1} x_{2p+1} 
\cos \Bigl( \dfrac{2 \pi k (2p+1) }{n} \Bigr)= \nonumber
\\ &=& \sum_{p=0}^{m-1} x_{2p} 
\cos \Bigl( \dfrac{2 \pi  k p }{m} \Bigr) +
 \sum_{p=0}^{m-1} x_{2p+1} 
\cos \Bigl( \dfrac{\pi k (2p+1) }{m} \Bigr) \nonumber 
\\ &=& \sum_{p=0}^{m-1} x^{\prime}_{p} 
\cos \Bigl( \dfrac{2 \pi  k p }{m} \Bigr) +
 \sum_{p=0}^{m-1} x^{\prime \prime}_{p} 
\cos \Bigl( \dfrac{2 \pi k p }{m} +
\dfrac{\pi k }{m}\Bigr) = \\ &=&  \nonumber
 \sum_{p=0}^{m-1} x^{\prime}_{p} 
\cos \Bigl( \dfrac{2 \pi  k p }{m} \Bigr) +\\ &+&
 \sum_{p=0}^{m-1} x^{\prime \prime}_{p} 
\Bigl( \cos \Bigl( \dfrac{2 \pi k p }{m} \Bigr) \cos
\Bigl(\dfrac{\pi k }{m}\Bigr) - 
\sin \Bigl( \dfrac{2 \pi k p }{m} \Bigr) \sin
\Bigl(\dfrac{\pi k }{m}\Bigr) \Bigr) \nonumber
= \\ &=&  \nonumber
 \sum_{p=0}^{m-1} x^{\prime}_{p} 
\cos \Bigl( \dfrac{2 \pi  k p }{m} \Bigr) +\\ &+& \cos
\Bigl(\dfrac{\pi k }{m}\Bigr) 
\sum_{p=0}^{m-1} x^{\prime \prime}_{p} 
\cos \Bigl( \dfrac{2 \pi k p }{m}\Bigr)  - \sin
\Bigl(\dfrac{\pi k }{m}\Bigr) 
\sum_{p=0}^{m-1} x^{\prime \prime}_{p}
\sin \Bigl( \dfrac{2 \pi k p }{m} \Bigr) = \\ &=& 
{\textrm{C}}_k(\eta \, {\mathbf{x}})
+ \cos \Bigl( \dfrac{\pi \, k}{m} \Bigr) \, 
{\textrm{C}}_k(\zeta \, {\mathbf{x}}) - \sin \Bigl( \dfrac{\pi \, k}{m} \Bigr)
\, {\textrm{S}}_k(\zeta \, {\mathbf{x}}). \nonumber
\end{eqnarray*} 
We now turn to (\ref{zksin}).
Using (\ref{SIN}), for $k=1,2, \ldots, \nu$ we get
\begin{eqnarray*}
{\textrm{S}}_k(\mathbf{x})&=&
\sum_{j=0}^{n-1} x_j 
\sin \Bigl(2 \pi j - \dfrac{2 \pi k j}{n} \Bigr)= \sum_{j=0}^{n-1} x_j 
\sin \Bigl(\dfrac{2 \pi k j}{n} \Bigr)=
\\ &=& \nonumber \sum_{j \, \text{even}} x_j 
\sin \Bigl( \dfrac{2 \pi k j }{n} \Bigr) + \sum_{j \, 
\text{odd}} x_j \sin \Bigl( \dfrac{2 \pi  k j}{n} \Bigr)=
\\ &=& \sum_{p=0}^{m-1} x_{2p} 
\sin \Bigl( \dfrac{4 \pi  k p }{n} \Bigr) +
 \sum_{p=0}^{m-1} x_{2p+1} 
\sin \Bigl( \dfrac{2 \pi k (2p+1) }{n} \Bigr)= 
\\&=&  \sum_{p=0}^{m-1} x^{\prime}_{p} 
\sin \Bigl( \dfrac{2 \pi  k p }{m} \Bigr) +\nonumber
 \sum_{p=0}^{m-1} x^{\prime \prime}_{p} 
\sin \Bigl( \dfrac{2 \pi k p }{m} +
\dfrac{\pi k }{m}\Bigr) = \\ &=&  \nonumber
 \sum_{p=0}^{m-1} x^{\prime}_{p} 
\sin \Bigl( \dfrac{2 \pi  k p }{m} \Bigr) + \\&+&\nonumber
 \sum_{p=0}^{m-1} x^{\prime \prime}_{p} 
\Bigl( \sin \Bigl( \dfrac{2 \pi k p }{m} \Bigr) \cos
\Bigl(\dfrac{\pi k }{m}\Bigr) + 
\cos \Bigl( \dfrac{2 \pi k p }{m} \Bigr) \sin
\Bigl(\dfrac{\pi k }{m}\Bigr) \Bigr)
= \\ &=&  \nonumber
 \sum_{p=0}^{m-1} x^{\prime}_{p} 
\sin \Bigl( \dfrac{2 \pi  k p }{m} \Bigr) +  \cos
\Bigl(\dfrac{\pi k }{m}\Bigr)  
\sum_{p=0}^{m-1} x^{\prime \prime}_{p} 
\sin \Bigl( \dfrac{2 \pi k p }{m}\Bigr)  + \\&+&\sin
\Bigl(\dfrac{\pi k }{m}\Bigr) 
\sum_{p=0}^{m-1} x^{\prime \prime}_{p}
\cos \Bigl( \dfrac{2 \pi k p }{m} \Bigr)= \nonumber
\\ &=& {\textrm{S}}_k(\eta \, {\mathbf{x}})
+ \cos \Bigl( \dfrac{\pi \, k}{m} \Bigr) \, 
{\textrm{S}}_k(\zeta \, {\mathbf{x}}) + \sin \Bigl( \dfrac{\pi \, k}{m} \Bigr)
\, {\textrm{C}}_k(\zeta \, {\mathbf{x}}) .
\end{eqnarray*}
\end{proof}
\begin{lemma}\label{coseno2}
For $k=0,1,\ldots, \nu-1$, let $t=m-k$. 
We have
\begin{eqnarray}\label{zkcos0000}
{\rm{C}}_t(\mathbf{x})={\rm{C}}_k(\eta \, \mathbf{x})
- \cos \Bigl( \dfrac{\pi \, k}{m} \Bigr) \, 
{\rm{C}}_k(\zeta \,\mathbf{x}) + \sin \Bigl( \dfrac{\pi \, k}{m} \Bigr)
\, {\rm{S}}_k(\zeta \,\mathbf{x}).
\end{eqnarray}
Furthermore, for
$k=1,2, \ldots, \nu-1$
we get
\begin{eqnarray}\label{zksin2}
{\rm S}_t(\mathbf{x})={\rm{S}}_k(\eta \,\mathbf{x})
- \cos \Bigl( \dfrac{\pi \, k}{m} \Bigr) \, 
{\rm{S}}_k(\zeta \, \mathbf{x})+ \sin \Bigl( \dfrac{\pi \, k}{m} \Bigr)
\, {\rm{C}}_k(\zeta \,\mathbf{x}).
\end{eqnarray}
\end{lemma}

\begin{proof}  We begin with (\ref{zkcos0000}).
For $t=\nu+1, \, \nu +2, \ldots, m$, we have
\begin{eqnarray*}\label{cost}
{\rm{C}}_t(\mathbf{x})&=& \nonumber
\sum_{p=0}^{m-1} x_{2p} 
\cos \Bigl( \dfrac{4 \pi  t p }{n} \Bigr) +
\sum_{p=0}^{m-1} x_{2p+1} 
\cos \Bigl( \dfrac{2 \pi t (2p+1) }{n} \Bigr)= \\ &=&
\nonumber
\sum_{p=0}^{m-1} x_{2p} 
\cos \Bigl( \dfrac{4 \pi  (m-k) p }{n} \Bigr) +
\sum_{p=0}^{m-1} x_{2p+1} 
\cos \Bigl( \dfrac{2 \pi (m-k) (2p+1) }{n} \Bigr)= \\ &=&
\sum_{p=0}^{m-1} x^{\prime}_{p} 
\cos \Bigl( 2 \pi  p- \dfrac{2 \pi  k p }{m} \Bigr) +
\sum_{p=0}^{m-1} x^{\prime \prime}_{p} 
\cos \Bigl( \pi (2p+1) -
\dfrac{\pi k (2p+1)}{m}\Bigr) \nonumber = \\ &=& 
\sum_{p=0}^{m-1} x^{\prime}_{p} 
\cos \Bigl( \dfrac{2 \pi  k p }{m} \Bigr) -
\sum_{p=0}^{m-1} x^{\prime \prime}_{p}
\cos \Bigl( \dfrac{\pi k (2p+1)}{m}\Bigr)= \\&=& \nonumber
\sum_{p=0}^{m-1} x^{\prime}_{p} 
\cos \Bigl( \dfrac{2 \pi  k p }{m} \Bigr) - \cos
\Bigl(\dfrac{\pi k }{m}\Bigr) 
\sum_{p=0}^{m-1} x^{\prime \prime}_{p} 
\cos \Bigl( \dfrac{2 \pi k p }{m}\Bigr)  + \\ &+& \sin
\Bigl(\dfrac{\pi k }{m}\Bigr) 
\sum_{p=0}^{m-1} x^{\prime \prime}_{p}
\sin \Bigl( \dfrac{2 \pi k p }{m} \Bigr)= \nonumber
\\&=& \nonumber {\rm C}_k(\eta \, {\mathbf{x}}) - \cos
\Bigl(\dfrac{\pi k }{m}\Bigr) 
{\rm C}_k(\zeta \,{\mathbf{x}})  + \sin
\Bigl(\dfrac{\pi k }{m}\Bigr) 
{\rm S}_k(\zeta  \, {\mathbf{x}}). 
\nonumber
\end{eqnarray*}
Now we turn to (\ref{zksin2}).
For $t=\nu+1,  \, \nu +2, \ldots, m-1$, it is
\begin{eqnarray*}\label{sint}
{\rm S}_t(\mathbf{x})&=& \nonumber
\sum_{p=0}^{m-1} x_{2p} 
\sin \Bigl( \dfrac{4 \pi  t p }{n} \Bigr) +
\sum_{p=0}^{m-1} x_{2p+1} 
\sin \Bigl( \dfrac{2 \pi t (2p+1) }{n} \Bigr)= \\ &=&
\nonumber
\sum_{p=0}^{m-1} x_{2p} 
\sin \Bigl( \dfrac{4 \pi  (m-k) p }{n} \Bigr) +
\sum_{p=0}^{m-1} x_{2p+1} 
\sin \Bigl( \dfrac{2 \pi (m-k) (2p+1) }{n} \Bigr)= \\ &=&
\sum_{p=0}^{m-1} x^{\prime}_{p} 
\sin \Bigl( 2 \pi  p- \dfrac{2 \pi  k p }{m} \Bigr) +
\sum_{p=0}^{m-1} x^{\prime \prime}_{p} 
\sin \Bigl( \pi (2p+1) -
\dfrac{\pi k (2p+1)}{m}\Bigr) \nonumber =
 \\ &=& 
- \sum_{p=0}^{m-1} x^{\prime}_{p} 
\sin \Bigl( \dfrac{2 \pi  k p }{m} \Bigr) +
\sum_{p=0}^{m-1} x^{\prime \prime}_{p}
\sin \Bigl( \dfrac{\pi k (2p+1)}{m}\Bigr)= \\&=& \nonumber
-\sum_{p=0}^{m-1} x^{\prime}_{p} 
\sin \Bigl( \dfrac{2 \pi  k p }{m} \Bigr) + \sin
\Bigl(\dfrac{\pi k }{m}\Bigr) 
\sum_{p=0}^{m-1} x^{\prime \prime}_{p} 
\cos \Bigl( \dfrac{2 \pi k p }{m}\Bigr)  + \\ &+& \cos
\Bigl(\dfrac{\pi k }{m}\Bigr) 
\sum_{p=0}^{m-1} x^{\prime \prime}_{p}
\sin \Bigl( \dfrac{2 \pi k p }{m} \Bigr)= \nonumber
\\&=& \nonumber -{\rm S}_k(\eta \, {\mathbf{x}}) + \sin
\Bigl(\dfrac{\pi k }{m}\Bigr) 
{\rm C}_k(\zeta \, {\mathbf{x}}) + \cos
\Bigl(\dfrac{\pi k }{m}\Bigr) 
{\rm S}_k(\zeta \, {\mathbf{x}}). \nonumber
\end{eqnarray*}
\end{proof}
Note that \begin{eqnarray}\label{zkcos0}
{\rm{C}}_0(\mathbf{x})={\rm{C}}_0(\eta \, \mathbf{x})
+{\rm{C}}_0(\zeta \, \mathbf{x}),
\end{eqnarray}
\begin{eqnarray}\label{zkcosparticular}
{\rm{C}}_m(\mathbf{x})={\rm{C}}_0(\eta \, \mathbf{x})
- {\rm{C}}_0(\zeta \, \mathbf{x}).
\end{eqnarray}
Since ${\rm{S}}_{\nu}(\mathbf{x})=0$ whenever $\mathbf{x}
\in {\mathbb{R}}^m$, then
\begin{eqnarray}\label{zkcos111}
{\rm{C}}_{\nu}(\mathbf{x})={\rm{C}}_{\nu}(\eta \, \mathbf{x})
- \, {\rm{S}}_{\nu}(\zeta \, \mathbf{x})=
{\rm{C}}_{\nu}(\eta \, \mathbf{x})
\end{eqnarray} and
\begin{eqnarray}\label{zksinparticular}
{\rm S}_{\nu}(\mathbf{x})={\rm{S}}_{\nu}(\eta \, \mathbf{x})
+ {\rm{C}}_{\nu}(\zeta \, \mathbf{x})= {\rm{C}}_{\nu}(\zeta 
\, \mathbf{x}).
\end{eqnarray}
We call $\rho:\mathbb{R}^{n} \to \mathbb{R}^{m}$ 
that function which reverses 
all components of a given vector $\mathbf{x}
\in \mathbb{R}^n $ but the 
$0$-th component, and $\sigma$ (resp., $\alpha$):
$\mathbb{R}^n \to \mathbb{R}^n$ 
that function which associates to every vector
$\mathbf{x}
\in \mathbb{R}^n $ the double 
of its symmetric (resp., asymmetric) part, namely
\begin{eqnarray*}\label{trick0}
\rho(\mathbf{x}) &=& (x_0 \, x_{n-1} \, \nonumber
x_{n-2} 
 \cdots x_2 \, x_1 )^T,  \\
\sigma  (\mathbf{x})=\sigma \, \mathbf{x}  &=& 
{\mathbf{x}}
+\rho(\mathbf{x}), \, \,  \quad
\alpha  (\mathbf{x}) = \alpha\, \mathbf{x} =
{\mathbf{x}}-\rho(\mathbf{x}).  \nonumber
\end{eqnarray*} 
Note that 
\begin{eqnarray}\label{symmasymm}
{\mathbf{x}}=
\dfrac{\sigma \,\mathbf{x}+\alpha \, \mathbf{x}}{2}.
\end{eqnarray}
%

Now we state the next technical results.
%
\begin{lemma}\label{matemagica}
Let $\mathbf{a} = 
(a_0 \, a_1 \, \cdots a_{n-1} )^T$,
$\mathbf{b} = 
(b_0 \, b_1 \, \cdots b_{n-1} )^T$
$\in \mathbb{R}^n$
be such that $\mathbf{a}$ is symmetric and 
$\mathbf{b}$ is asymmetric.
Then, $\mathbf{a}^T \, \mathbf{b}=0$.
\end{lemma}
\begin{proof} 
First of all, we observe that 
$b_0=b_{n/2}=0$. So, we have
\begin{eqnarray*}\label{scalarproductzero3}
\mathbf{a}^T \, \mathbf{b}&=& \sum_{j=0}^{n-1} 
a_j \, b_j=\sum_{j=1}^{n-1} 
a_j \, b_j= \sum_{j=1}^{n/2-1} 
a_j \, b_j + a_{n/2}\, b_{n/2} + \sum_{j=n/2+1}^n
a_j \, b_j = \\ &=& \sum_{j=1}^{n/2-1} 
a_j \, b_j + \sum_{j=1}^{n/2-1} a_{n-j} \, b_{n-j}=
\sum_{j=1}^{n/2-1} a_j \, b_j - \sum_{j=1}^{n/2-1} 
a_j \, b_j =0.  \nonumber 
\end{eqnarray*} 
\end{proof}
\begin{corollary} 
Let ${\mathbf x} \in {\mathbb R}^{n}$ be a symmetric vector
and $k \in \{ 0, 1, \ldots, n-1\}$.
Then, we get
\begin{eqnarray}\label{sinez}
{\rm S}_k(\mathbf{x})=0. 
\end{eqnarray} 
\end{corollary}
\begin{proof}
Note that ${\rm S}_k(\mathbf{x})= {\mathbf{x}}^T \,
{\mathbf{v}^{(k)}}$.
The formula (\ref{sinez}) follows from Lemma \ref{matemagica},
since $\mathbf{x}$ is symmetric and 
${\mathbf{v}^{(k)}}$ is asymmetric.  
\end{proof}
\begin{corollary} 
Let ${\mathbf x} \in {\mathbb R}^{n}$ be an asymmetric vector
and $k \in \{ 0, 1, \ldots, n-1\}$.
Then, we have
\begin{eqnarray}\label{cosinez000}
{\rm C}_k(\mathbf{x})=0. 
\end{eqnarray} 
\end{corollary}
\begin{proof}
Observe that ${\rm C}_k(\mathbf{x})= \mathbf{x}^T
{\mathbf{u}^{(k)}}$.
The formula (\ref{cosinez000}) 
is a consequence of Lemma \ref{matemagica},
because ${\mathbf{u}^{(k)}}$ is symmetric
and $\mathbf{x}$ is asymmetric.  
\end{proof}
\begin{lemma} 
Let ${\mathbf x} \in {\mathbb R}^{n}$
and $k \in \{ 0, 1, \ldots, m\}$. Then we get
\begin{eqnarray}\label{cosineeven}
C_k(\sigma \, \mathbf{x})=2 \, C_k(\mathbf{x}).
\end{eqnarray}
Moreover, if
$k \in \{ 1, 2, \ldots, m-1\}$, then 
\begin{eqnarray}\label{sineodd}
{\rm S}_k(\alpha \, \mathbf{x})=2 \, S_k(\mathbf{x}).
\end{eqnarray}
\end{lemma}
\begin{proof}
We begin with proving  (\ref{cosineeven}). 
For every $j \in \{0, 1, \ldots, n -1 \}$ we have
\begin{eqnarray*}
C_k(\sigma \, \mathbf{x}) &=& 2 x_0 +\sum_{j=1}^{n-1} x_j u^{(k)}_j +
\sum_{j=1}^{n-1} x_{n-j} u^{(k)}_j =
\\ &=& 2 x_0 + \sum_{j=1}^{n-1} x_j u^{(k)}_j +
2 \sum_{j=1}^{n-1} x_j u^{(k)}_j =
2 \sum_{j=0}^{n-1} x_j u^{(k)}_j =
2 \, C_k(\mathbf{x}).
\end{eqnarray*}
Now we turn to (\ref{sineodd}). 
For any $j \in \{0, 1, \ldots, n -1 \}$,
since $v^{(k)}_0=0$, we get
\begin{eqnarray*}
S_k(\alpha \, \mathbf{x}) &=& \sum_{j=1}^{n-1} x_j 
v^{(k)}_j -
\sum_{j=1}^{n-1} x_{n-j} v^{(k)}_j =\sum_{j=1}^{n-1} x_j v^{(k)}_j +
\sum_{j=1}^{n-1} x_{n-j} v^{(k)}_{n-j}= \\&=&
2 \sum_{j=1}^{n-1} x_j v^{(k)}_j =
2 \sum_{j=0}^{n-1} x_j v^{(k)}_j =
2 \, S_k(\mathbf{x}). 
\end{eqnarray*}
\end{proof}

Observe that from (\ref{IDSCT}), (\ref{symmasymm}), (\ref{sinez})
and (\ref{cosinez000}), taking into account that 
$\sigma \, \mathbf{x}$ is symmetric and 
$\alpha \, \mathbf{x}$ is asymmetric, we obtain
\begin{eqnarray}\label{IDSCTbis}
\textrm{IDSCT}(\mathbf{x})&=&Q_n^T \, \mathbf{x}
=
\left(
\begin{array}{c}
{\alpha_0} \, \dfrac{{\textrm{C}}_0(\sigma \, \mathbf{x})
}{2} \\ \\ {\alpha_1} \, \dfrac{{\textrm{C}}_1(\sigma \, \mathbf{x})
}{2} \\ \vdots \\
{\alpha_m} \, \dfrac{{\textrm{C}}_m(\sigma \, \mathbf{x})
}{2} \\ \\
{\alpha_{m+1}} \, \dfrac{
{\textrm{S}}_{m-1}(\alpha \, \mathbf{x})}{2} \\  \vdots \\ 
{\alpha_{n-1}} \, \dfrac{
{\textrm{S}}_1(\alpha \, \mathbf{x})}{2}
\end{array}\right) .
\end{eqnarray}
Thus, to compute IDSCT($\mathbf{x}$),
we determine the values 
${\rm C}_k(\sigma \, \mathbf{x})$ for $k=0,1,\ldots, m$ 
and ${\rm S}_k(\alpha \, \mathbf{x})$ for $k=1,2,\ldots, m-1$,
and successively we multiply such values by
the constants $\alpha_{k}/2$, $k=0,1,\ldots, n-1$.
Now we give the following
\begin{lemma}\label{instrcos}
Let ${\mathbf{x}} \in {\mathbb{R}}^n$ be symmetric. Then
for $k=1, 2,\ldots, \nu-1$ we get
\begin{eqnarray}\label{cosinfty}
{\rm C}_k ({\mathbf{x}})=
{\rm C}_k (\eta \, {\mathbf{x}})+\dfrac{1}{2 \, \cos 
\Bigl( \dfrac{\pi \, k }{m} \Bigr)} \, \, {\rm C}_k 
(\sigma \, \zeta \, {\mathbf{x}}).
\end{eqnarray}
Moreover, if $t=m-k$, then
\begin{eqnarray}\label{new}
{\rm C}_t(\mathbf{x})=
{\rm C}_k(\eta \, {\mathbf{x}})- \dfrac{1}{2 \, \cos
\Bigl(\dfrac{\pi k }{m}\Bigr) }
\,{\rm C}_k(\sigma \, \zeta \, {\mathbf{x}}).
\end{eqnarray}
\end{lemma}
\begin{proof}
We begin with (\ref{cosinfty}).
Since $\mathbf{x}$ is symmetric, we have
\begin{eqnarray}\label{00s}
{\rm S}_k(\mathbf{x})={\rm{S}}_k(\eta \, {\mathbf{x}})=0.
\end{eqnarray}
From (\ref{zksin}) and (\ref{00s}) we obtain 
\begin{eqnarray}\label{provvdz}
{\rm S}_k(\zeta \, {\mathbf{x}})= - \dfrac{\sin
\Bigl(\dfrac{\pi k }{m}\Bigr)}{\cos\Bigl(\dfrac{\pi k }{m}\Bigr)} 
\,{\rm C}_k(\zeta \, {\mathbf{x}}).
\end{eqnarray}
From (\ref{zkcos}), (\ref{cosineeven}) and (\ref{provvdz}) 
we get (\ref{cosinfty}). The equality in
(\ref{new}) follows from (\ref{zkcos0000}), (\ref{cosineeven}) and (\ref{provvdz}).  
\end{proof}

Now we define the following algorithm:
\begin{algorithmic}[1]
\STATE{{\bf{function CS$({\bf x},n)$}}}
\IF{$n=4$}
\STATE{$c_0={x}_0+ 2 {x}_1 + {x}_2$;}
\STATE{$c_1={x}_0 - x_2;$}
\STATE{$c_2=x_0 - 2  {x}_1 +  {x}_2$;}
\ELSE
\STATE{$\widetilde{\bf c}=$CS$(\eta \, {\bf x}, 
n/2)$;} 
\STATE{$\overline{\bf c}=$CS$(\sigma
\, \zeta \, {\bf x}, 
 n/2)$;} 
\FOR{$k=1, 2, \ldots n/4-1$}
\STATE{$aux=1/(2 \cos (2 \, \pi \, k /n)) \, \overline{c}_k$;}
\STATE{$c_k= \widetilde{c}_k + aux$;}
\STATE{$c_{n/2-k}= \widetilde{c}_k - aux$;}
\ENDFOR
\STATE{$aux= \overline{c}_0/2$;}
\STATE{$c_0=\widetilde{c}_0 + aux$;}
\STATE{$c_{n/4}=\widetilde{c}_{n/4}$;}
\STATE{$c_{n/2}=\widetilde{c}_0 - aux$;}
\ENDIF
\RETURN{${\bf c}$}
\end{algorithmic}
\begin{lemma}\label{presymm}
For each symmetric vector ${\mathbf{x}} \in{\mathbb{R}}^n$,
the vector $\eta \, {\mathbf{x}} \in {\mathbb{R}}^m$ is symmetric.
\end{lemma}
\begin{proof}
Let $\widetilde{\mathbf{x}}=\eta \, {\mathbf{x}}$, where ${\mathbf{x}}$ 
is a symmetric vector.
So, we have
\begin{eqnarray}\label{symmsymm}
\widetilde{x}_j={{x}}_{2j}={{x}}_{n-2j}={{x}}_{2(n/2-j)}=\widetilde{x}_{n/2-j},
\quad j=0,1,\ldots m.
\end{eqnarray}
Since $\widetilde{\mathbf{x}} \in {\mathbb{R}}^m$, from 
(\ref{symmsymm}) 
it follows that $\widetilde{\mathbf{x}}$ is symmetric.
 
\end{proof}
\begin{theorem}  \label{tomove}
Let $\mathbf{x} \in {\mathbb{R}}^n$ be a symmetric vector and
${\bf c}=$
{\rm CS}$({\mathbf{x}}, n)$. Then,
\begin{eqnarray*}\label{yjcoseno}
c_j={\rm C}_j(\mathbf{x})=\mathbf{x}^T {\mathbf{u}}^{(j)},
\quad j=0,1,\ldots, m.
\end{eqnarray*}
\end{theorem}
\begin{proof}
First of all we prove that, when we call the function CS,
the first variable is symmetric.
Concerning the external call, the vector $\mathbf{x}$ 
is symmetric in CS. Now we consider the internal calls 
of CS at lines 7 
and 8.
By Lemma \ref{presymm}, $\eta \, {\mathbf{x}}$
is symmetric. 
Hence, line 7 of our algorithm satisfies the assertion.
Furthermore, it is readily seen that the assertion is fulfilled also by line 8.

Now, let us suppose that $n=4$. Then, 
\begin{eqnarray*}\label{c0}
c_0= {\mathbf{x}}^T \, {\mathbf{u}}^{(0)} \, =
x_0 + x_1 + x_2 + x_3.
\end{eqnarray*}
Since ${\mathbf{x}}$ is even, we get the formula at line 3 of 
the algorithm CS. Moreover, 
\begin{eqnarray*}
{\mathbf{u}}^{(1)} = (1 \, \, \, \, \, \, 0 \, \, -1 \, \, \,\, \,  \, 0)^T,
\end{eqnarray*}
and hence 
\begin{eqnarray*}
c_1= {\mathbf{x}}^T \, {\mathbf{u}}^{(1)} \, =
{{x}}_0 - {{x}}_2 ,
\end{eqnarray*}
obtaining the formula at line 4 of 
the algorithm CS. Furthermore, 
\begin{eqnarray*}
{\mathbf{u}}^{(2)} = (1 \, \, -1 \, \, \, \, \, \, 1 \, \, -1)^T,
\end{eqnarray*}
and thus we have
\begin{eqnarray*}\label{c2}
c_2= {\mathbf{x}}^T \, {\mathbf{u}}^{(2)} \, =
{{x}}_0 - 2 \, {{x}}_1 + {{x}}_2 ,
\end{eqnarray*}
since ${\mathbf{x}}$ is symmetric. So, we obtain
the formula at line 5 of 
the algorithm CS.

Now we consider the case $n >4$. Since ${\mathbf{x}}$ is symmetric,
from Lemma \ref{instrcos} we deduce that line 11
of the algorithm CS gives the correct value 
of ${\rm C}_k({\mathbf{x}})$
for $k=1$, $2, \ldots, \nu -1$. As ${\mathbf{x}}$ is symmetric,
from Lemma \ref{instrcos} we obtain that line 12 gives the exact value 
of ${\rm C}_{m-k}({\mathbf{x}})$ for $k=1$, $2, \ldots, \nu -1$.

From (\ref{zkcos0}) and (\ref{cosineeven}) we deduce that line 15
of the function CS gives the correct value of 
${\rm C}_0({\mathbf{x}})$, and by (\ref{zkcosparticular})
and (\ref{cosineeven}) we obtain
that line 17 of the function CS gives the exact value of 
${\rm C}_m({\mathbf{x}})$. Furthermore, by virtue of
(\ref{zkcos111}), line 16 of the function CS gives the correct value of
${\rm{C}}_{\nu}({\mathbf{x}})$.
This completes the proof.  
\end{proof}
\begin{theorem}\label{tomove3}
Suppose to have a library in which the values
$1/(2 \cos (2 \, \pi \, k /n))$ have been stored 
for $n=2^r$, $r \in\mathbb{N}$,
$r \geq 2$, and $k \in \{ 0$, $1, \ldots, n-1\}$.
Then, the computational cost of a call of the function {\rm CS} 
is given by 
\begin{eqnarray}\label{lawadd0bis}
A(n)=\dfrac34 n \, \log_2 n - \dfrac12 n +1,
\end{eqnarray}
\begin{eqnarray}\label{lawmult0bis}
M(n)= \dfrac14 n\, \log_2 n+ \dfrac12 n  -2,
\end{eqnarray}
where $A(n)$ \rm (\em resp., $M(n)$\rm) \em denotes the number of
additions \rm(\em resp., multiplications\rm) \em requested to compute {\rm CS}.
\end{theorem}
\begin{proof}
Concerning line 8, we first compute
$\sigma \, \zeta \, {\mathbf{x}}$. To do this, 
$\nu -1$ sums and $2$ multiplications are required. 
To compute the {\bf for} loop at lines 9-13, 
$2\, \nu - 2$ sums and $\nu - 1 $ multiplications
are necessary. Moreover, to compute lines 14-17,
2 sums and one multiplication are necessary.
Thus, the total number of additions 
is given by 
\begin{eqnarray}\label{ricorsivabis}
A(n)=\dfrac34 n - 1 +2  \, A\Bigl(\dfrac{n}{2} \Bigr),
\end{eqnarray} while
the total number of multiplications
is 
\begin{eqnarray}\label{ricorsivamultbis}
M(n)=\dfrac14 n  +2 +2  \, M\Bigl(\dfrac{n}{2} \Bigr).
\end{eqnarray} Concerning the initial case, we have
\begin{eqnarray*}\label{ricorsivainitialbis}
A(4)=5, \quad M(4)=2.
\end{eqnarray*}

For every $n=2^r$, with $r \in \mathbb{N}$, $r \geq 2$, 
let $A(n)$ be as in (\ref{lawadd0bis}).
We get $A(4)=6-2+1=5$. Now we claim that 
the function $A(n)$ defined in (\ref{lawadd0bis}) satisfies
(\ref{ricorsivabis}). Indeed, we have
\begin{eqnarray*}\label{lawadd2bis}
A\Bigl(\dfrac{n}{2} \Bigr)= \dfrac38 n \,(\log_2 n -1)
-\dfrac14 n +1,
\end{eqnarray*}
and hence
\begin{eqnarray*}\label{ricorsivayesbis} 
\dfrac34 n - 1 +2  \, A\Bigl(\dfrac{n}{2} \Bigr)=
\dfrac34 n - 1+  \dfrac34 n \,(\log_2 n -1)
- \dfrac12 n +2=A(n), 
\end{eqnarray*}
getting the claim.

Moreover, for any $n=2^r$, with $r \in \mathbb{N}$, $r \geq 2$,  
let $M(n)$ be as in (\ref{lawmult0bis}).
One has $M(4)=2+ 2-2=2$. Now we claim that 
the function $M(n)$ defined in (\ref{lawmult0bis}) fulfils
(\ref{ricorsivamultbis}). It is
\begin{eqnarray*}\label{lawmult2bis}
M\Bigl(\dfrac{n}{2} \Bigr)= \dfrac18 n \,(\log_2 n -1)+ \dfrac14n-2,
\end{eqnarray*}
and hence
\begin{eqnarray*}\label{ricorsivayeesbis}
\dfrac14 n +2+2  \, M\Bigl(\dfrac{n}{2} \Bigr)=
\dfrac14 n +2 +\dfrac14 n \,(\log_2 n -1)+ \dfrac12 n-4=M(n),
\end{eqnarray*}
obtaining the claim.  
\end{proof}
\begin{lemma}\label{instrsin}
Let ${\mathbf{x}} \in {\mathbb{R}}^n$ be asymmetric. Then,
for $k=1, 2,\ldots, \nu-1$ it is
\begin{eqnarray}\label{sinkwelatnom} 
{\rm S}_k(\mathbf{x})= 
{\rm S}_k(\eta \, {\mathbf{x}}) + \dfrac{1}{2 \, \cos
\Bigl(\dfrac{\pi k }{m}\Bigr)}
\,{\rm S}_k(\alpha \, \zeta \, {\mathbf{x}}).
\end{eqnarray}
Moreover, if $t=m-k$, then
\begin{eqnarray}\label{sinkkt}
{\rm S}_t(\mathbf{x})=
- {\rm S}_k(\eta \, {\mathbf{x}}) + \dfrac{1}{2 \, \cos
\Bigl(\dfrac{\pi k }{m}\Bigr)} 
\,{\rm S}_k(\alpha \, \zeta \, {\mathbf{x}}).
\end{eqnarray}
\end{lemma}
\begin{proof}
We begin with (\ref{sinkwelatnom}).
As $\mathbf{x}$ is asymmetric, we have
\begin{eqnarray}\label{00}
{\rm C}_k(\mathbf{x})={\rm{C}}_k(\eta \, {\mathbf{x}})=0.
\end{eqnarray}
From (\ref{zkcos}) and (\ref{00}) we deduce 
\begin{eqnarray}\label{elatnom0}
{\rm C}_k(\zeta \, {\mathbf{x}}) = \dfrac{\sin
\Bigl(\dfrac{\pi k }{m}\Bigr)
}{\cos
\Bigl(\dfrac{\pi k }{m}\Bigr)} 
\, {\rm S}_k(\zeta \, {\mathbf{x}}).
\end{eqnarray}
From (\ref{zksin}), (\ref{sineodd}) and (\ref{elatnom0}) 
we get (\ref{sinkwelatnom}). 

The relation (\ref{sinkkt})
follows from (\ref{zksin2}), (\ref{sineodd}) and (\ref{elatnom0}).  
\end{proof}
Now we define the following algorithm:
\begin{algorithmic}[1]
\STATE{{\bf{function SN$({\bf x},n)$}}}
\IF{$n=4$}
\STATE{$s_1=2 \, {x}_1$;}
\ELSE
\STATE{$\widetilde{\bf s}=$SN$( 
\eta \, {\bf x},n/2)$;} 
\STATE{$\overline{\bf s}=$SN$( 
\alpha \, \zeta \, {\bf x}, n/2)$;} 
\FOR{$k=1, 2, \ldots n/4-1$}
\STATE{$aux=1/(2 \cos (2 \, \pi \, k /n)) \, \overline{s}_k$;}
\STATE{$s_k= \widetilde{s}_k + aux$;}
\STATE{$s_{n/2-k}=aux  -\widetilde{s}_k$;}
\ENDFOR
\STATE{$s_{n/4}=0$;}
\FOR{$j=0,2,\ldots, n/4-2$}
\STATE{$s_{n/4}=s_{n/4}+{x}_{2j+1}- {x}_{2j+3}$}
\ENDFOR
\STATE{$s_{n/4}= 2 \, s_{n/4}$}
\ENDIF
\RETURN{${\bf s}$}
\end{algorithmic}

\begin{lemma}\label{preasymm}
For every asymmetric vector ${\mathbf{x}} \in{\mathbb{R}}^n$,
the vector $\eta \, {\mathbf{x}} \in {\mathbb{R}}^m$ is asymmetric.
\end{lemma}
\begin{proof}
Let $\overline{\mathbf{x}}=\eta \, {\mathbf{x}}$, where ${\mathbf{x}}$ 
is asymmetric. Then,
\begin{eqnarray}\label{asymmasymm}
\overline{x}_j={{x}}_{2j}=-{{x}}_{n-2j}=-{{x}}_{2(n/2-j)}=-\overline{x}_{n/2-j},
\quad j=0,1,\ldots m.
\end{eqnarray}
Since $\widetilde{\mathbf{x}} \in {\mathbb{R}}^m$, from 
(\ref{asymmasymm}) 
it follows that $\overline{\mathbf{x}}$ is asymmetric.  
\end{proof}
\begin{lemma}\label{presymmasymm}
Given any asymmetric vector $\mathbf{x} \in {\mathbb{R}}^n$, 
set $\widehat{\mathbf{x}}=\zeta \, {\mathbf{x}}$. Then,
\begin{eqnarray}\label{mcgplrcg}
{\rm S}_{\nu} (
\widehat{\mathbf{x}})=2 \sum_{j=0}^{\nu-1} (-1)^j \, 
\widehat{x}_{2j+1}.
\end{eqnarray}
\end{lemma}
\begin{proof}
Let ${\mathbf{x}}\in 
{\mathbb{R}}^n$ be any asymmetric vector. Then, 
$\widehat{\mathbf{x}}=\zeta \, {\mathbf{x}}
\in{\mathbb{R}}^m$. It is not difficult to check that
\begin{eqnarray}\label{barbar}
\widehat{x}_j=- \widehat{x}_{n-1-j}, 
\quad j=0,1,\ldots n-1.
\end{eqnarray}
Since ${\mathbf{u}}^{(\nu)}=( 1 \, \, -1 \, \, \ldots 1 \, \, -1)^T$
whenever ${\mathbf{u}}^{(\nu)}\in 
{\mathbb{R}}^m$, the equality in (\ref{mcgplrcg}) follows
from (\ref{zksinparticular})
and (\ref{barbar}).  
\end{proof}
\begin{theorem}  \label{tomove2}
Given an asymmetric vector $\mathbf{x} \in {\mathbb{R}}^n$,
let 
${\bf s}=$
{\rm SN}$({\mathbf{x}}, n)$. We get 
\begin{eqnarray*}\label{yjseno}
s_j={\rm S}_j(\mathbf{x})=\mathbf{x}^T {\mathbf{v}}^{(j)},
\quad j=1,2,\ldots, m-1.
\end{eqnarray*}\end{theorem}
\begin{proof}
First we prove that, when we call the function SN,
the first variable is asymmetric.
Concerning the external call, we have that $\mathbf{x}$ 
is asymmetric. 
Now we consider the internal calls 
of SN at lines 5
and 6.
By Lemma \ref{preasymm}, 
$\eta \, \mathbf{x}$ is asymmetric.
Therefore, line 5 of our algorithm fulfils the assertion.
Moreover, it is easy to see that line 6 satisfies the assertion too.

Now, let us suppose that $n=4$. We have
\begin{eqnarray*}
{\mathbf{v}}^{(1)} = (0 \, \,\, \, \, \,  1 \, \, \, \, \,  \, 0 \, \, -1)^T,
\end{eqnarray*} and hence
\begin{eqnarray*}
s_1= {\mathbf{x}}^T \, {\mathbf{v}}^{(1)} \, =
{{x}}_1 - {{x}}_3 = 2 {{x}}_1 ,
\end{eqnarray*}
because ${\mathbf{x}}$ is asymmetric.
Therefore, we obtain the formula at line 3 of 
the algorithm SN.

Now we consider the case $n >4$. As ${\mathbf{x}}$ is asymmetric,
from Lemma \ref{instrsin} we get that line 9
of the function SN gives the right value 
of ${\rm S}_k({\mathbf{x}})$
for $k=1$, $2, \ldots, \nu -1$. Since ${\mathbf{x}}$ is asymmetric,
from Lemma \ref{instrsin} we obtain that 
line 10 of the function SN gives the exact value 
of ${\rm S}_{m-k}({\mathbf{x}})$ for $k=1$, $2, \ldots, \nu -1$.
Thanks to Lemma \ref{presymmasymm} and asymmetry of
${\mathbf{x}}$, lines 12-16 give the exact value of 
${\rm{S}}_{\nu}({\mathbf{x}})$.
This ends the proof.  
\end{proof}
\begin{theorem}
Suppose to have a library in which the values
$1/(2 \cos (2 \, \pi \, k /n))$ have been stored 
for $n=2^r$, $r \in\mathbb{N}$,
$r \geq 2$, and $k \in \{ 0$, $1, \ldots, n-1\}$.
Then, the computational cost of a call of the function {\rm SN} 
is given by 
\begin{eqnarray}\label{lawadd0}
A(n)= n \, \log_2 n -\dfrac{11}{4} n +3,
\end{eqnarray}
\begin{eqnarray}\label{lawmult0}
M(n)= \dfrac14 n\, \log_2 n-\dfrac14 n ,
\end{eqnarray}
where $A(n)$ \rm (\em resp., $M(n)$\rm) \em denotes the number of
additions \rm(\em resp., multiplications\rm) \em required to compute {\rm CS}.
\end{theorem}
\begin{proof}
Concerning line 6, we first compute
$\alpha \, \zeta \, {\mathbf{x}}$. To do this, 
$\nu -1$ sums and no multiplications are required.
To compute the {\bf for} loop at lines 7-11, 
$2\, \nu - 2$ sums and $\nu -1 $ multiplications
are necessary. 
Concerning the {\bf for} loop at lines 13-15,
$\nu$ sums are necessary. Finally, at line 16,
one multiplication 
is necessary. Thus, the total number of 
additions is  
\begin{eqnarray}\label{ricorsivann}
A(n)=n - 3 +2  \, A\Bigl(\dfrac{n}{2} \Bigr),
\end{eqnarray}
and the total number of multiplications is 
\begin{eqnarray}\label{ricorsivamultnn}
M(n)=\dfrac14 n +2  \, M\Bigl(\dfrac{n}{2} \Bigr).
\end{eqnarray} Concerning the initial case, we have
\begin{eqnarray*}\label{ricorsivainitialnn}
A(4)=0, \quad M(4)=1.
\end{eqnarray*}
For each $n=2^r$, with $r \in \mathbb{N}$, $r \geq 2$, 
let $A(n)$ be as in (\ref{lawadd0}).
We get $A(4)=8-11+3=0$. 
Now we claim that 
the function $A(n)$ defined in (\ref{lawadd0}) satisfies
(\ref{ricorsivann}). Indeed, we have
\begin{eqnarray*}\label{lawadd2nn}
A\Bigl(\dfrac{n}{2} \Bigr)=-\dfrac{11}{8} n + \dfrac12 n \,(\log_2 n -1)+3.
\end{eqnarray*}
Therefore,
\begin{eqnarray*}\label{ricorsivayesnnh} 
n - 3 +2  \, A\Bigl(\dfrac{n}{2} \Bigr)=
n - 3  -\dfrac{11}{4} n + n \,(\log_2 n -1)+6=A(n),
\end{eqnarray*}
which gives the claim.

Moreover, for every $n=2^r$, with $r \in \mathbb{N}$, $r \geq 2$,  
let $M(n)$ be as in (\ref{lawmult0}).
One has $M(4)=2-1=1$. Now we claim that 
the function $M(n)$ defined in (\ref{lawmult0}) fulfils
(\ref{ricorsivamultnn}). Indeed, 
\begin{eqnarray*}\label{lawmult2}
M\Bigl(\dfrac{n}{2} \Bigr)=-\dfrac{1}{8} n + \dfrac18 n \,(\log_2 n -1),
\end{eqnarray*}
and hence
\begin{eqnarray*}\label{ricorsivayees}
\dfrac14 n +2  \, M\Bigl(\dfrac{n}{2} \Bigr)=
\dfrac14 n - \dfrac14 n+\dfrac14 n \,(\log_2 n -1)=M(n),
\end{eqnarray*}
getting the claim.  
\end{proof}

\begin{corollary}  
Let  $\mathbf{x} \in {\mathbb{R}}^n$,
${\mathbf{y}}=${\rm IDSCT}$({\mathbf{x}})$, 
${\bf c}=$
{\rm CS}$(\sigma \, {\mathbf{x}}, n)$ and
${\bf s}=$
{\rm SN}$(\alpha \, {\mathbf{x}}, n)$. Then,
\begin{eqnarray}\label{yj1}
y_j=\left\{ \begin{array}{ll}
\dfrac{\alpha_j \, c_j}{2}, & \text{   if  } j \leq m; \\ \\
\dfrac{\alpha_j \, s_{n-j}}{2}, & \text{   otherwise.  }
\end{array} \right.
\end{eqnarray}
\end{corollary}
\begin{proof}
Since $\sigma \, \mathbf{x}$ 
is symmetric, from Theorem \ref{tomove}
we obtain
\begin{eqnarray*}\label{yjcosinus}
c_j={\rm C}_j(\sigma \, \mathbf{x}),
\quad j=0,1,\ldots, m.
\end{eqnarray*}
Moreover, as $\alpha \, \mathbf{x}$ 
is asymmetric, from Theorem \ref{tomove2}
we get 
\begin{eqnarray*}\label{yjsinus}
s_j={\rm S}_j(\alpha \, \mathbf{x}),
\quad j=1,2,\ldots, m-1.
\end{eqnarray*}
Therefore, from (\ref{IDSCTbis}) we deduce (\ref{yj1}).  
\end{proof}

To complete the computation 
of IDSCT$(\mathbf{x})$ as in (\ref{IDSCTbis}), we have to compute 
$\sigma \, \mathbf{x}$ and $\alpha \, \mathbf{x}$ and to 
multiply every entry of the result of CS$(\sigma \, \mathbf{x},
\alpha \, \mathbf{x},n)$ by $\alpha_{k}/2$, $k=0,1,\ldots, n-1$.
The computational cost of these operations is $O(n)$. 
Therefore, the total cost for the computation of 
(\ref{2.2}) is $\frac74 \, n \, \log_2 n+o( n \, \log_2 n)$ additions and
$\frac12 \, n \, \log_2 n+o( n \, \log_2 n)$ multiplications.

\subsection{Computation of the eigenvalues of a
$\gamma$-matrix}
In this subsection
we find an algorithm to compute the expression in (\ref{2.3}).
We first determine the eigenvalues of the matrix $G$ in 
(\ref{primordiale}), in order to know
the matrix $\Lambda^{(G)}$ in (\ref{2.3}). 

Since $G \in {\mathcal G}_n$ and 
${\mathcal G}_n={\mathcal C}_n \oplus {\mathcal B}_n$, 
we determine two matrices $C \in {\mathcal C}_n$
and $B \in {\mathcal B}_n$, $C
=$circ$(\mathbf{c})$, $B
=$rcirc$(\mathbf{b})$.
Observe that
\begin{eqnarray}\label{firstdecomp}
g_{0,0}&=&c_0+b_0, \qquad
g_{\nu,\nu}=c_0+b_m, \\
g_{0,m}&=&c_m+b_m, \qquad
g_{\nu,m+\nu}=c_m+b_0. \nonumber
\end{eqnarray}
By solving the system in (\ref{firstdecomp}), we find 
$c_0$, $b_0$, 
$c_m$, $b_m$. 
Knowing $c_0$, from the principal diagonal of 
$G$ it is not difficult to determine the numbers 
$b_{2i}$, $i=1$, $2, \ldots, \nu-1$.
If we know these quantities, it is possible
to determine the numbers
${c}_{2i}$, $i=1$, $2, \ldots, \nu-1$,
from the first row of $G$.
Moreover, note that
\begin{eqnarray}\label{seconddecomp}
g_{0,1}&=&{c}_1+{b}_1, \qquad
g_{1,2}=c_1+{b}_3, \\
g_{0,3}&=&{c}_3+{b}_3, \qquad
g_{m-1,m+2}={c}_3+{b}_1. \nonumber
\end{eqnarray}
By solving the system in (\ref{seconddecomp}), we obtain
$c_1$, $b_1$, 
$c_3$, $b_3$. 
If we know $c_1$, then from the first diagonal 
superior to the principal diagonal of 
$G$ it is not difficult to determine the numbers 
${b}_{2i+1}$, $i=2$, $3, \ldots, \nu-1$.
Knowing these values, it is not
difficult to find the quantities
${c}_{2i+1}$, $i=2$, $3, \ldots, \nu-1$,
from the first row of $G$.
It is possible to prove that the computational cost
of all these operations is $O(n)$.
Note that \begin{eqnarray}\label{eigensum}
G \, \mathbf{q}^{(j)} &=& C
\, \mathbf{q}^{(j)} + B
\, \mathbf{q}^{(j)} =\left({\lambda}^{(C)}_j +
{\lambda}_j^{(B)} \right) \, \mathbf{q}^{(j)}, \quad 
j=0,1, \ldots, n-1,
\end{eqnarray}
where the ${\lambda}^{(C)}_j $'s and the
${\lambda}^{(B)}_j $'s are the $j$-th eigenvalues 
of $C$ and $B$, and the orders 
are given by Theorems \ref{Theorem3.1} and
\ref{Theorem3.126},
respectively.
\begin{proposition}  
Given two symmetric vectors
$\bf c$, $\bf b\in {\mathbb{R}}^n$
such that $$\displaystyle{\sum_{t=0}^{n-1} b_t = 0} \quad
\text{   and   } \quad
\displaystyle{\sum_{t=0}^{n-1} (-1)^t  b_t
\, = \, 0},$$
set ${\bf d}^{(C)}=$
{\rm CS}$({\bf c}, n)$ and
${\bf d}^{(B)}=$
{\rm CS}$({\bf b}, n)$.
Then the eigenvalues of 
$G={\rm circ}({\bf c})+{\rm rcirc}({\bf b})$ are given by
\begin{eqnarray}\label{yj}
\lambda^{(G)}_0&=&d_0^{(C)};  \qquad 
\lambda^{(G)}_j=d^{(C)}_j + d^{(B)}_j ,
\quad j=1,2,\ldots m-1;\\  
\lambda_m^{(G)}
&=&d_m^{(C)}; \qquad 
\lambda_j^{(G)}=d^{(C)}_{n-j} - d^{(B)}_{n-j},
\quad j=m+1,m+2,\ldots n-1. \nonumber
\end{eqnarray}
\end{proposition}
\begin{proof} 
Since $\mathbf{c}$ and $\mathbf{b}$ 
are symmetric, from Theorem \ref{tomove}
we obtain
\begin{eqnarray*}
d^{(C)}_j={\rm C}_j({\mathbf{c}})=
{\mathbf{c}}^T  {\mathbf{u}}^{(j)},
\quad
d^{(B)}_j={\rm C}_j({\mathbf{b}})=
{\mathbf{c}}^T  {\mathbf{u}}^{(j)},
\quad j=0,1,\ldots, m.
\end{eqnarray*}
By Theorems \ref{Theorem3.1} and
\ref{Theorem3.126}, we get that $ \lambda^{(C)}_j=
{{d}}^{(C)}_j$ and $ \lambda^{(B)}_j=
d^{(B)}_j$
for $j=0$, $1, \ldots, m$, where ${{\lambda}}^{(C)}_j$
(resp., ${{\lambda}}^{(B)}_j$) are the eigenvalues of
${\rm circ}(\mathbf{c})$ 
(resp., ${\rm rcirc}(\mathbf{b})$).
From Theorems \ref{Theorem3.1},
\ref{Theorem3.126} and (\ref{eigensum}) we obtain
(\ref{yj}). 
 
\end{proof}

To complete the computation of (\ref{2.3}), we have to multiply the diagonal matrix
$\Lambda^{(G)}$ by $\mathbf{z}$. The cost of this operation consists of
$n$ multiplications. Thus, by Theorem \ref{tomove3},
the total cost to compute
(\ref{2.3}) is of $\frac32 \, n \,\log_2 n +o(n \,\log_2 n)$ additions and
$\frac12 \, n \,\log_2 n +o(n \,\log_2 n)$ multiplications.
\subsection{The DSCT technique}
Now we compute $\mathbf{y}=Q_n \, \mathbf{t}=
{\rm DSCT}(\mathbf{t})$. Let 
$\overline{\alpha}$ and 
$\widetilde{\alpha}$
be as in (\ref{alphak}).
By (\ref{prestartingmatrix}), 
for $j=0$, $1, \ldots, m$ it is
\begin{eqnarray}\label{i} \nonumber
y_j&=& \overline{\alpha} \, t_0 +(-1)^j \,\overline{\alpha} \,
t_{n/2}+\widetilde{\alpha} \,\sum_{k=1}^{m-1} 
\cos \Bigl( \dfrac{2 \, k \, \pi \, j}{n}\Bigl) \, t_k +\\ &+&
\widetilde{\alpha} \,\sum_{k=1}^{m-1} 
\sin \Bigl( \dfrac{2 \, k \, \pi \, j}{n}\Bigl) \, t_{n-k}=
{\rm DSCT}_j(\mathbf{t}),
\end{eqnarray}
and for $j=1$, $2, \ldots, m-1$ we have
\begin{eqnarray}\label{n-i} \nonumber
y_{n-j}&=& \overline{\alpha} \, t_0 +(-1)^j \,\overline{\alpha} \,
t_{n/2}+\widetilde{\alpha} \,\sum_{k=1}^{m-1} 
\cos \Bigl( \dfrac{2 \, k \, \pi \,(n-j)}{n}\Bigl) \, t_k +\\ &+&
\widetilde{\alpha} \,\sum_{k=1}^{m-1} 
\sin \Bigl( \dfrac{2 \, k \, \pi \, (n-j)}{n}\Bigl) \, t_{n-k}=\\ &=&
\nonumber \overline{\alpha} \, t_0 +(-1)^j \,\overline{\alpha} \,
t_{n/2}+\widetilde{\alpha} \,\sum_{k=1}^{m-1} 
\cos \Bigl( \dfrac{2 \, k \, \pi \, j}{n}\Bigl) \, t_k -\\ &-&
\widetilde{\alpha} \,\sum_{k=1}^{m-1} 
\sin \Bigl( \dfrac{2 \, k \, \pi \, j}{n}\Bigl) \, t_{n-k}=
{\rm DSCT}_{n-j}(t). \nonumber
\end{eqnarray}
Now we define the functions $\varphi$, $\vartheta 
: \mathbb{R}^n \to \mathbb{R}^n$ by
\begin{eqnarray}\label{t1}
\varphi({\mathbf{t}})   &=& \overline{\mathbf{t}}
=(t_0 \, \,  t_1  \, \, \ldots t_{m-1} \, \, t_m \, \, 
 t_{m-1} \, \, \ldots \, \, t_1 ),    	\\
  \vartheta({\mathbf{t}})  &=&  \widetilde{\mathbf{t}}
=(0 \, \, t_{n-1}\, \, t_{n-2} \, \, \ldots t_{m+1} \, \,0
\, \, -t_{m+1} \, \, \ldots \, \, - t_{n-1} ). \nonumber
\end{eqnarray}
The following result holds.
\begin{theorem}\label{3.24}
For $j=1$, $2, \ldots, m-1$, it is
\begin{eqnarray}\label{j}
{\rm DSCT}_j(\mathbf{t})
&=& \dfrac{\widetilde{\alpha}}{2} ({\rm C}_j 
(\overline{\mathbf{t}})-t_0 - (-1)^j t_m +  
{\rm S}_j(\widetilde{\mathbf{t}}) )+ \\ &+& \overline{\alpha} \, 
t_0 +(-1)^j \,\overline{\alpha} \,
t_m \nonumber
\end{eqnarray}
and
\begin{eqnarray}\label{nminusj}
{\rm DSCT}_{n-j}(\mathbf{t})&=& \dfrac{\overline{\alpha}}{2} ({\rm C}_j 
(\overline{\mathbf{t}})-t_0 - (-1)^j t_m -  
{\rm S}_j(\widetilde{\mathbf{t}}) )+ \\ &+& \overline{\alpha} \, 
t_0 +(-1)^j \,\overline{\alpha} \,
t_m. \nonumber
\end{eqnarray}
Moreover, one has
\begin{eqnarray}\label{0}
{{\rm DSCT}_0}(\mathbf{t})&=& \dfrac{\widetilde{\alpha}}{2} ({\rm C}_0 
(\overline{\mathbf{t}})-t_0 - t_m )+\overline{\alpha} \, 
t_0 +\overline{\alpha} \,
t_m; 
\end{eqnarray}
\begin{eqnarray}\label{m}
{{\rm DSCT}_m}
(\mathbf{t})&=& \dfrac{\widetilde{\alpha}}{2} ({\rm C}_m
(\overline{\mathbf{t}})-t_0 - t_m) +\overline{\alpha} \, 
t_0 +\overline{\alpha} \,
t_m.
\end{eqnarray}
\end{theorem}
\begin{proof}
We have
\begin{eqnarray} \label{coscos} \nonumber
 {\rm C}_j 
(\overline{\mathbf{t}})&=&{\overline{\mathbf{t}}}^T
{\mathbf{u}}^{(j)}=\overline{t}_0+(-1)^j \overline{t}_m+
\sum_{k=1}^{m-1} 
\cos \Bigl( \dfrac{2 \, k \, \pi \, j}{n}\Bigl) \, \overline{t}_k +
\sum_{k=m+1}^{n-1} 
\cos \Bigl( \dfrac{2 \, k \, \pi \, j}{n}\Bigl) \, \overline{t}_k =\\ 
&=& \nonumber  t_0+(-1)^j \, t_m+
\sum_{k=1}^{m-1} 
\cos \Bigl( \dfrac{2 \, k \, \pi \, j}{n}\Bigl) \, t_k +
\sum_{k=1}^{m-1} 
\cos \Bigl( \dfrac{2 \, (n-k) \, \pi \, j}{n}\Bigl) \, 
\overline{t}_{n-k} = \\&=&  t_0+(-1)^j t_m+
\sum_{k=1}^{m-1} 
\cos \Bigl( \dfrac{2 \, k \, \pi \, j}{n}\Bigl) \, t_k +
\sum_{k=1}^{m-1} 
\cos \Bigl( \dfrac{2 \, (n-k) \, \pi \, j}{n}\Bigl) \, t_k =
\\&=&  t_0+(-1)^j t_m+2
\sum_{k=1}^{m-1} 
\cos \Bigl( \dfrac{2 \, k \, \pi \, j}{n}\Bigl) \, t_k ; \nonumber
\end{eqnarray}
\begin{eqnarray} \label{sinsin}
\nonumber  {\rm S}_j 
(\widetilde{\mathbf{t}}) &=&{\widetilde{\mathbf{t}}}^T
{\mathbf{v}}^{(j)}=
\sum_{k=1}^{m-1} 
\sin \Bigl( \dfrac{2 \, k \, \pi \, j}{n}\Bigl) \, \widetilde{t}_k +
\sum_{k=m+1}^{n-1} 
\sin \Bigl( \dfrac{2 \, k\, \pi \, j}{n}\Bigl) \, \widetilde{t}_k 
=\\ 
&=& 
\sum_{k=1}^{m-1} 
\sin \Bigl( \dfrac{2 \, k \, \pi \, j}{n}\Bigl) \, t_{n-k} -
\sum_{k=1}^{m-1} 
\sin \Bigl( \dfrac{2 \, (n-k) \, \pi \, j}{n}\Bigl) \, 
t_{n-k} = \\&=& \nonumber 2 
\sum_{k=1}^{m-1} 
\sin \Bigl( \dfrac{2 \, k \, \pi \, j}{n}\Bigl) \, t_{n-k}.
\end{eqnarray}
From (\ref{i}) and (\ref{coscos}) we obtain (\ref{j}), 
(\ref{0}) and (\ref{m}), while from (\ref{n-i}) and (\ref{sinsin}) we 
get 
(\ref{nminusj}).  
\end{proof}
The computational cost of the valuation of the functions
 $\varphi$ and $\vartheta$ is
linear, and the cost of the call of the functions 
CS and SN to determinate the values C$_j(\overline{\mathbf{t}})$
for $j=0$, $1, \cdots, m$ and S$_j(\widetilde{\mathbf{t}})$
for $j=1$, $2, \cdots, m$ is of 
$\frac74 \, n \,\log_2 n +o(n \,\log_2 n)$ additions and
$\frac12 \, n \,\log_2 n +o(n \,\log_2 n)$ multiplications.
The remaining computations for DSCT$(\mathbf{t})$
are of $O(n)$. Thus, the total cost to compute 
DSCT$(\mathbf{t})$ is of 
$\frac74 \, n \,\log_2 n +o(n \,\log_2 n)$ additions and
$\frac12 \, n \,\log_2 n +o(n \,\log_2 n)$ multiplications.
Therefore, the total cost to compute (\ref{primordiale}) is given by 
$5\, n \,\log_2 n +o(n \,\log_2 n)$ additions and
$\frac32 \, n \,\log_2 n +o(n \,\log_2 n)$ multiplications.
\section{Multiplication between two $\gamma$-matrices}
In this section we show how to compute the product of 
two given $\gamma$-matrices. 
\begin{lemma}
Let $G_1$, $G_2$ be two $\gamma$-matrices. Then,
the first column of 
$\Lambda^{(G_1)}\, \Lambda^{(G_2)} Q_n^T$
is given by 

\begin{eqnarray}\label{l}
\mathbf{s}^{(G_1,G_2)}=
\begin{pmatrix} 
\overline{\alpha} \,\lambda_0^{(G_1)} \, 
\lambda_0^{(G_2)}\\ \\
\widetilde{\alpha} \,\lambda_1^{(G_1)} \, 
\lambda_1^{(G_2)}\\ \vdots \\
\widetilde{\alpha} \,\lambda_{{n/2}-1}^{(G_1)} \, 
\lambda_{{n/2}-1}^{(G_2)}\\ 
\\
\overline{\alpha} \,\lambda_{n/2}^{(G_1)} \, 
\lambda_{n/2}^{(G_2)}\\ \\ 0 \\ \vdots \\0
\end{pmatrix},
\end{eqnarray}
where $\overline{\alpha}$ and 
$\widetilde{\alpha}$ are as in \rm (\ref{alphak}) \em
and $\lambda_j^{(G_1)}$ (resp., $\lambda_j^{(G_2)}$),
$j=0$, $1, \ldots, n/2$, are the first $n/2+1$
eigenvalues of $G_1$
(resp., $G_2$).
Moreover, 
$\mathbf{s}^{(G_1,G_2)}$ can be computed by 
$\frac32 \, n \log_2 n+
o( n \log_2 n)$ additions and
$\frac12 \, n \log_2 n +
o( n \log_2 n)$
multiplications.
\end{lemma}
\begin{proof}
First, we note that 
\begin{eqnarray}\label{prestartingmatrix0}
q^{(n)}_{1,j}=\left\{ \begin{array}{ll}
\overline{\alpha}, & \text{if  } j=0 \text{  or  }
j= n/2 ; \\  \\ \widetilde{\alpha},
 & \text{if  } j=1,2, \ldots, n/2 - 1; \\ \\ 0, & \text{otherwise.}
\end{array} \right.
\end{eqnarray} 
Thus, the first part of the 
assertion follows from (\ref{prestartingmatrix0}).

Moreover, we get
\begin{eqnarray*}
\lambda_j^{(G_1)}&=& {\mathbf{g_1}}^T \,
 \mathbf{u}^{(j)}={\rm C}_j(\mathbf{g_1}), 
\quad j =0,1, \ldots, m; \\ 
\lambda_j^{(G_2)}&=& {\mathbf{g_2}}^T\,
 \mathbf{u}^{(j)}={\rm C}_j(\mathbf{g_2}), 
\quad j =0,1, \ldots, m,
\end{eqnarray*} 
where $\mathbf{g_1}$ and $\mathbf{g_2}$ are the first row
of $G_1$ and $G_2$, respectively.

In order to compute the eigenvalues of $G_1$ and 
$G_2$, we do two calls of the function CS, obtaining
a computational cost of 
$\dfrac32 \, n \log_2 n+
o( n \log_2 n)$ additions of
$\dfrac12 \, n \log_2 n +
o( n \log_2 n)$
multiplications.
Note that the final multiplication cost
is not relevant, because it is of type 
$O(n)$.  
\end{proof}
\begin{theorem}
Given two $\gamma$-matrices $G_1$, $G_2$, where 
$G_1=C_1+B_1$, $G_2=C_2+B_2$, 
$C_l\in {\mathcal C}_n$, $B_l\in {\mathcal B}_n$, $l=1.2$, we have
\begin{eqnarray}\label{4760}
G_1\,  G_2 = {\rm{circ}}\left(Q_n \left(\mathbf{s}^{(C_1 C_2)}
+\mathbf{s}^{(B_1 B_2)}\right)\right) +
{\rm{rcirc}}\left(Q_n \left(\mathbf{s}^{(C_1 B_2)}
+\mathbf{s}^{(B_1 C_2)}\right)  \right).
\end{eqnarray}
Moreover, with the exception of the computation of 
the functions {\rm{circ}} and {\rm{rcirc}},
$G_1 \, G_2$ can be computed by means of
$\frac92 \, n \log_2 n+
o( n \log_2 n)$ additions of
$\frac32 \, n \log_2 n +
o( n \log_2 n)$
multiplications.
\end{theorem}
\begin{proof}
First, we note that 
\begin{eqnarray*}\label{distributive0}
G_1\, G_2= C_1 \, C_2 + C_1 \, B_2 +
B_1 \, C_2 + B_1 \, B_2. 
\end{eqnarray*}
From this it follows that the first column of $G_1\, G_2$ is equal to 
\begin{eqnarray}\label{QnQn}
Q_n \left(\mathbf{s}^{(C_1 C_2)} + 
\mathbf{s}^{(B_1 B_2)} \right) +
Q_n \left(\mathbf{s}^{(C_1 B_2)} + 
\mathbf{s}^{(B_1 C_2)} \right).
\end{eqnarray}
Moreover, observe that every $\gamma$-matrix 
has the property that its first column coincide
with the transpose of its first row. From this and 
(\ref{productm}) we obtain (\ref{4760}).

Furthermore note that, when we compute the two DSCT's in 
(\ref{QnQn}), from (\ref{l}) it follows that the vector
$\widetilde{\mathbf{t}}$ in $(\ref{t1})$ is equal to $0$ 
in both cases. Therefore, two calls of the function CS are needed.

Again, we observe that to compute
$\mathbf{s}^{(C_1 C_2)}$, $
\mathbf{s}^{(B_1 B_2)}$, 
$Q_n (\mathbf{s}^{(C_1 B_2)})$ and 
$\mathbf{s}^{(B_1 C_2)}$ it is enough
to compute the eigenvalues of $C_l$, $B_l$, $l=1,2$, only once.
Thus, four calls of the function CS are needed.
Hence, $G_1\, G_2$ can be computed with
$\frac92 \, n \log_2 n+
o( n \log_2 n)$ additions of
$\frac32 \, n \log_2 n +
o( n \log_2 n)$
multiplications.

\end{proof} 
\section{Toeplitz matrix preconditioning}\label{TT}
For each $n \in \mathbb{N}$, let us consider the following class:
\begin{eqnarray}\label{toeplitzs}
{\mathcal T}_n =\{T_n \in \mathbb{R}^{n \times n}:
T_n=(t_{k,j})_{k,j}, t_{k,j}=t_{|k-j|}, \, \, k,j \in \{0, 1, \ldots, n-1\}
\, \, \}. 
\end{eqnarray}
Observe that the class defined in (\ref{toeplitzs})
coincides with the family of all real symmetric Toeplitz matrices.
\vspace{3mm}

Now we consider the following problem.
\begin{description}
\item[] Given $T_n \in {\mathcal T}_n$, find
\begin{eqnarray*}\label{minmin}
G_n(T_n)=\min_{G \in {\mathcal G}_n} \|G - T_n \|_F,
\end{eqnarray*}
where $G_n(T_n) = C_n(T_n) + B_n(T_n)$,
$C_n(T_n) \in {\mathcal C}_n$,
$B_n(T_n) \in {\mathcal B}_n$, and $\|\cdot\|_F$ denotes 
the Frobenius norm.
\end{description}
\begin{theorem}\label{calculus} 
Let $\widehat{\mathcal G}_n= \mathcal{S}_n + 
{\mathcal{H}}_{n,1}$.
Given $T_n \in {\mathcal T}_n$, one has
\begin{eqnarray*}
C_n (T_n)+ B_n (T_n)= \min_{G\in \widehat{\mathcal G}_n } \|G- T_n \|_F=
\min_{G \in {\mathcal G}_n} \|G - T_n \|_F,
\end{eqnarray*}
where  $C_n(T_n)=${\rm{circ}}$({\mathbf{c}})$,
with
\begin{eqnarray}\label{cj}
c_j=\dfrac{(n-j)\, t_j + j \, t_{n-j}}{n}, \quad
j \in \{1,2, \ldots, n-1\}; 
\end{eqnarray}
\begin{eqnarray}\label{c0}
c_0=t_0,
\end{eqnarray}
and $B_n(T_n)=${\rm{rcirc}}$({\mathbf{b}})$,
where: 
for $n$ even and $j \in \{1, 2, \ldots, n-1\}
\setminus \{n/2\}$,
\begin{eqnarray}\label{evenodd}
b_j&=& \dfrac{1}{2n} \left( \dfrac{4 \,j-2 \, n}{n}
(t_j - t_{n-j}) + 4 \sum_{k=1}^{(j-3)/2}\dfrac{2\,k+1}{n} 
(t_{2k+1}-t_{n-2k-1}) +\right. \nonumber
\\& +& \left. 4 \sum_{k=1}^{(n-j-3)/2}\dfrac{2\,k+1}{n}
(t_{2k+1}-t_{n-2k-1}) \right), \quad j \text{ odd; } 
\end{eqnarray}
\begin{eqnarray}\label{eveneven}
b_j&=& \dfrac{1}{2n} \left( \dfrac{4 \,j-2 \, n}{n}
(t_j - t_{n-j}) + 4 \sum_{k=1}^{j/2-1}\dfrac{2\,k}{n} 
(t_{2k}-t_{n-2k}) +\right. \nonumber
\\& +& \left. 4 \sum_{k=1}^{(n-j)/2-1}\dfrac{2\,k}{n}
(t_{2k}-t_{n-2k}) \right), \quad j \text{  even;  }
\end{eqnarray}
for $n$ even,
\begin{eqnarray}\label{even0}
b_0&=& \dfrac{2}{n} \left( \sum_{k=1}^{n/2-1} 
\dfrac{2\,k}{n} (t_{2k}-t_{n-2k}) \right),
\end{eqnarray}
\begin{eqnarray}\label{evenn2}
b_{n/2}&=& \dfrac{4}{n} \left( \sum_{k=1}^{n/4-1}
\dfrac{2\,k}{n} 
(t_{2k}-t_{n-2k})\right) ; 
\end{eqnarray}
for $n$ odd and $j \in \{1, 2,\ldots, n-1\}$,
\begin{eqnarray}\label{oddodd}
b_j&=& \dfrac{1}{2n} \left( \dfrac{4 \,j-2 \, n}{n}
(t_j - t_{n-j}) + 4 \sum_{k=0}^{(j-3)/2}\dfrac{2\,k+1}{n} 
(t_{2k+1}-t_{n-2k-1}) +\right. \nonumber
\\& +& \left. 4 \sum_{k=1}^{(n-j)/2-1}\dfrac{2\,k}{n}
(t_{2k}-t_{n-2k}) \right), \quad j \text{ odd; }
\end{eqnarray}
\begin{eqnarray}\label{oddeven}
b_j&=& \dfrac{1}{2n} \left( \dfrac{4 \,j-2 \, n}{n}
(t_j - t_{n-j}) + 4 \sum_{k=1}^{j/2-1}\dfrac{2\,k}{n} 
(t_{2k}-t_{n-2k}) +\right. \nonumber
\\& +& \left. 4 \sum_{k=0}^{(n-j-3)/2}\dfrac{2\,k +1}{n}
(t_{2k+1}-t_{n-2k-1}) \right), \quad j \text{  even;  } 
\end{eqnarray}
for $n$ odd,
\begin{eqnarray}\label{odd0}
b_0&=& \dfrac{2}{n} \left(\sum_{k=0}^{(n-3)/2} 
\dfrac{2\,k+1}{n} (t_{2k+1}-t_{n-2k-1}) \right).
\end{eqnarray}
\end{theorem}
\begin{proof}
Let us define 
\begin{eqnarray*}\label{phi}
\phi(\mathbf{c},\mathbf{b})=\|T_n- \textrm{circ}({\mathbf{c}})
- \textrm{circ}({\mathbf{b}})\|^2_F
\end{eqnarray*}
for any two symmetric vectors
$\mathbf{c}$, $\mathbf{b} \in \mathbb{R}^n$.
If $j \in \{1,2, \ldots, n-1\}$, then we get
\begin{eqnarray} \label{fourth0}
\dfrac{\partial \phi(\mathbf{c},\mathbf{b})}{\partial c_j}=
-4(n-j)\, t_j -4 \, j \, t_{n-j}+4\sum_{j=0}^{n-1} b_j
+4\,n \, c_j .
\end{eqnarray}
Furthermore, one has
\begin{eqnarray} \label{fifth0}
\dfrac{\partial \phi(\mathbf{c},\mathbf{b})}{\partial c_0}=
- 2 \, n \, t_0+ 2
\sum_{j=0}^{n-1} b_j + 2 \, n \, c_0 .
\end{eqnarray}

If $n$ is even and $j$ is odd, $j \in \{1, \ldots, n-1\}$, then,
since $c_{n-j}=c_j$,
we have
\begin{eqnarray}\label{bjeo} 
\dfrac{\partial \phi(\mathbf{c},\mathbf{b})}{\partial b_j}&=&
-2 \left( 2 \, t_j + 4  \sum_{k=0}^{(j-3)/2} t_{2k+1}+ 2 \, t_{n-j} + 
4 \sum_{k=0}^{(n-j-3)/2} t_{2k+1}- \right. \nonumber
\\ & & \left. -4  \nonumber
\sum_{k=0}^{n/4-1} c_{2k+1} - 2 \, n \, b_j
\right)= \\&=& -2 \left( 2(t_j-c_j) + 4 
\sum_{k=0}^{(j-3)/2} (t_{2k+1}-c_{2k+1}) + 2 (t_{n-j} - c_{n-j}) +
\right. \\ & & \left. +
4 \sum_{k=0}^{(n-j-3)/2} (t_{2k+1}-c_{2k+1})-
2 \,n \, b_j \right). \nonumber
\end{eqnarray}

If both $n$ and $j$ are even, $j \in \{1, 2, \ldots, n-1\}
\setminus \{n/2\}$, then,
by arguing analogously as in the previous case, we deduce
\begin{eqnarray} \label{bjee}
\dfrac{\partial \phi(\mathbf{c},\mathbf{b})}{\partial b_j}&=&
-2\left(2(t_j-c_j) + 4 \nonumber
\sum_{k=1}^{j/2-1} (t_{2k}-c_{2k}) + 2 (t_{n-j} - c_{n-j}) +
\right. \\ & & \left. +4 (t_0 - c_0)+
4 \sum_{k=1}^{(n-j)/2-1} (t_{2k}-c_{2k})-
2 \,n \, b_j \right).
\end{eqnarray}
Moreover, if $n$ is even, then one has
\begin{eqnarray}\label{bje0}
\dfrac{\partial \phi(\mathbf{c},\mathbf{b})}{\partial b_0}=
-2 \left(  t_0-c_0 +2 
\sum_{k=1}^{n/2-1} (t_{2k}-c_{2k}) -  n \, b_0
\right),
\end{eqnarray}
getting (\ref{even0}). Furthermore, for $n$ even, we have
\begin{eqnarray}\label{bjen2} 
\dfrac{\partial \phi(\mathbf{c},\mathbf{b})}{\partial b_{n/2}}=
-2 \left( 2(t_{n/2}-c_{n/2}) +
4 \sum_{k=1}^{n/4-1} (t_{2k}-c_{2k}) 
-  n  b_{n/2}
\right). 
\end{eqnarray}


Now, if both $n$ and $j$ are odd, $j \in \{0,1, \ldots, n-1\}$, then,
taking into account (\ref{cj}), we obtain
\begin{eqnarray}\label{bjoo} 
\dfrac{\partial \phi(\mathbf{c},\mathbf{b})}{\partial b_j}&=&
-2 \left( 2 \, (t_j -c_j)+ 4  \sum_{k=0}^{(j-3)/2} (t_{2k+1}
-c_{2k+1})+ 2 \, (t_{n-j}-c_{n-j}) + \right.\nonumber
\\ & & \left. +4 \sum_{k=1}^{(n-j)/2-1} (t_{2k}-c_{2k}) 
+2 \,( t_0 -  c_0) 
- 2 \, n \, b_j
\right). 
\end{eqnarray}

If $n$ is odd and $j$ is even, $j \in \{0,1, \ldots, n-1\}$, then 
we have
\begin{eqnarray}\label{bjoe} 
\dfrac{\partial \phi(\mathbf{c},\mathbf{b})}{\partial b_j}&=&
-2 \left( 2 \, (t_j -c_j) + 4  \sum_{k=1}^{j/2-1} 
(t_{2k}-c_{2k})+ 2 \, (t_{n-j} - c_{n-j})+\right. \\& & + \left.
4 \sum_{k=1}^{(n-j-3)/2} (t_{2k+1}-c_{2k+1})+  \nonumber
2 \, (t_0 - c_0 )
- 2 \, n \, b_j
\right).
\end{eqnarray}

Finally, for $n$ odd, 
one has
\begin{eqnarray}\label{bjo0} 
\dfrac{\partial \phi(\mathbf{c},\mathbf{b})}{\partial b_0}=
-2 \left(  t_0-c_0 +2 \sum_{k=0}^{(n-3)/2} 
(t_{2k+1} - c_{2k+1})
- n \, b_0 \right).
\end{eqnarray}

It is not difficult to see that the
function $\phi$ is convex. By \cite[Theorem 2.2]{DFZ},
$\phi$ has exactly one point of minimum. From this it follows 
that $\phi$ admits exactly one stationary point. 
Now we claim that this point satisfies 
%
\begin{eqnarray}\label{first0}
\sum_{k=0}^{n/4-1} b_{2k+1}=0
\end{eqnarray}
and 
\begin{eqnarray}\label{second0}
b_0+2\sum_{k=1}^{n/4-1} b_{2k}+b_{n/2}=0
\end{eqnarray}
when $n$ is even, and 
%
\begin{eqnarray}\label{third0}
b_0+2\sum_{j=1}^{(n-1)/2} b_j=0
\end{eqnarray}
if $n$ is odd, that is  $B_n(T_n) \in {\mathcal B}_n$.
From (\ref{first0})%
-(\ref{third0}) and
(\ref{fourth0})-(\ref{fifth0}) we get (\ref{cj})-(\ref{c0}).
Furthermore, from (\ref{cj})-(\ref{c0}) and
(\ref{bjeo})%
-(\ref{bjo0}) we obtain
(\ref{evenodd})%
-(\ref{odd0}). Finally, (\ref{first0})-(\ref{second0}) follow
from (\ref{evenodd})
-(\ref{evenn2}), while (\ref{third0}) is a consequence of 
(\ref{oddodd})
-(\ref{odd0}). 
\end{proof} 

Now, for $n \in \mathbb{N}$, set
\begin{eqnarray}\label{wienert}
\widehat{\mathcal T}_n &=&  
\{t \in {\mathcal T}_n: \text{  there is a function }
f(z)= \sum_{j=-\infty}^{+\infty} t_j \, z^j, \\ 
 & & \text{with  } z \in \mathbb{C},
|z|=1,
\text{  and  such  that   }
\sum_{j=-\infty}^{+\infty} |t_j| < + \infty  \}.\nonumber
\end{eqnarray}
Observe that any function defined by a power series
as in the first line of (\ref{wienert})
is real-valued, and the set of such functions 
satisfying the condition 
$$\sum_{j=-\infty}^{+\infty} |t_j| < + \infty $$
is called \emph{Wiener class} (see, e.g., 
\cite{BINIFAVATI1993}, 
\cite[\S3]{CHANetal}).

Given a function $f$ belonging to the Wiener class and
 a matrix $T_n \in \widehat{{\mathcal T}_n}$, 
$T_n(f)=(t_{k,j})_{k,j}$: $t_{k,j}=t_{|k-j|}$, $k,j \in \{0$,
$1, \ldots, n-1\}$, and $\displaystyle{f(z)= \sum_{j=-\infty}^{+\infty} t_j \, z^j}$, then we say that $T_n(f)$ is \emph{generated by $f$}.

We will often use the following property of 
absolutely convergent series (see, e.g., 
\cite{BINIFAVATI1993}, \cite{CHAN}).

\begin{lemma}\label{insp}
Let $\displaystyle{\sum_{j=1}^{\infty} t_j }$ be an
absolutely convergent series. Then, we get
\begin{eqnarray*}\label{inspproperty}
\lim_{n \to + \infty}\left[ \dfrac{1}{n} 
\left(\sum_{k=1}^n k \, |t_k|+
\sum_{k=\lceil(n+1)/2 \rceil}^n{(n-k)} \, |t_k|
\right)\right]=0.
\end{eqnarray*}
\end{lemma}
\begin{proof}
Let $\displaystyle{S=\sum_{j=1}^{\infty} |t_j| }$.
Choose arbitrarily $\varepsilon > 0$. By hypothesis,
there is a positive integer $n_0$ with 
\begin{eqnarray}\label{restseries}
\sum_{k=n_0+1}^{\infty} |t_k| \leq 
\dfrac{\varepsilon}{4}.
\end{eqnarray}
Let $\displaystyle{n_1=
\max\left\{\dfrac{2\, n_0 \,S}{\varepsilon}, 2 \, n_0\right\}}$.
Taking into account (\ref{restseries}),
for every $n > n_1$ it is
\begin{eqnarray*}
0&\leq& \dfrac{1}{n} \left(\sum_{k=1}^n k \, |t_k| +
\sum_{k=\lceil(n+1)/2 \rceil}^n{(n-k)} \, |t_k|\right)= \\&=&
\dfrac{1}{n} \sum_{k=1}^{n_0} k \, |t_k| +  \dfrac{1}{n}
\sum_{k=n_0+1}^{n} k \, |t_k|+ \dfrac{1}{n}
\sum_{k=\lceil(n+1)/2 \rceil}^{n} (n-k) \, |t_k|\leq 
\\ &\leq& \dfrac{1}{n_1} \, n_0
\sum_{k=1}^{n_0} |t_k| +  
2 \sum_{k=n_0+1}^{n} |t_k|\leq 
\dfrac{\varepsilon}{2 \, n_0 \, S} \, n_0 \, S + 2 \,
\dfrac{\varepsilon}{4}=\varepsilon.
\end{eqnarray*}
So, the assertion follows.
\end{proof}
\begin{theorem}\label{calculusestimate}
For $n \in \mathbb{N}$, given $T_n(f) \in 
\widehat{\mathcal T}_n$, let $C_n(f)=C_n(T_n(f))$,
$B_n(f)=B_n(T_n(f))$ be
as in Theorem \ref{calculus}, and set
$G_n(f)= C_n(f) + B_n(f)$ . Then, one has: 
\begin{description}
\item[\rm{\ref{calculusestimate}.1)}]
for every 
$\varepsilon >0$ there is a positive integer $n_0$, 
such that for each $n \geq n_0$ 
and for every eigenvalue
$\lambda_j^{(G_n(f))}$ of $G_n(f)$,
 it is
\begin{eqnarray}\label{propprop}
\lambda_j^{(G_n(f))} \in [f_{\rm{min}} - \varepsilon,
f_{\rm{max}} + \varepsilon], \quad j \in \{0,1, \ldots, n-1\},
\end{eqnarray}
where $f_{\rm{min}}$ and $f_{\rm{max}}$
denote the minimum and the maximum value of $f$, 
respectively;
\item[\rm{\ref{calculusestimate}.2)}] 
for every $\varepsilon > 0$ there are $k$, $n_1 \in 
\mathbb{N}$ such that for each $n \geq n_1$ 
the number of eigenvalues $\lambda_j^{((G_n(f))^{-1} \, T_n(f))}$ of 
$G_n^{-1} \, T_n(f)$ such that $|\lambda_j^{((G_n(f))^{-1} \, T_n(f))} - 1 | > \varepsilon$
is less than $k$, namely the spectrum of $(G_n(f))^{-1} \, T_n(f)$ is clustered around $1$.
\end{description}
\end{theorem}
\begin{proof} We begin with proving \ref{calculusestimate}.1). 
Choose arbitrarily $\varepsilon >0$. We denote by 
$\lambda_j^{(C_n(f))}$ (resp., $\lambda_j^{(B_n(f))}$)
the generic $j$-th eigenvalue 
of $C_n(f)$ (resp., $B_n(f)$), in the order given by 
Theorem \ref{Theorem3.1} (resp., Theorem \ref{Theorem3.126}).
To prove the assertion it is enough to show that
property (\ref{propprop}) holds 
(in correspondence with $\varepsilon/2$) for each
$\lambda_j^{(C_n(f))}$, $j =0$, $1, \ldots, n-1$, and that 
\begin{eqnarray}\label{lambd}
\lambda_j^{(B_n(f))} \in [-\varepsilon/2,\varepsilon/2] 
\text{   for  every   } 
n \geq n_0 \text{  and  } j\in \{0,1, \ldots, n-1\} .
\end{eqnarray}
Indeed, since $C_n(f)$, $B_n(f) \in  
{\mathcal G}_n$, we have
$$\lambda_j^{(G_n(f))}=\lambda_j^{(C_n(f))}+\lambda_j^{(B_n(f))}
\text{  for  all  }j \in \{0, 1, \ldots, n-1\},
$$ getting the assertion.

We first consider the case $n$ odd.
For every $j \in \{0, 1, \ldots, n-1\}$, 
since $c_j=c_{n-j}$ and thanks to (\ref{cj}), one has
\begin{eqnarray}\label{ruota}   \nonumber
\left|\lambda_j^{(C_n(f))}\right|&=&\left|\sum_{h=0}^{n-1} c_h
\cos\left( 2 \, \pi \, h \, j \right) \right|=\left|
c_0 + 2 \sum_{h=1}^{(n-1)/2}c_h
\cos\left( 2 \, \pi \, h \, j \right) \right|=\\&=& \left|
t_0 + 2 \sum_{h=1}^{(n-1)/2}t_h \nonumber
\cos\left( 2 \, \pi \, h \, j \right) \right.-
\\ &-& \left. \sum_{h=1}^{(n-1)/2}
\frac{h}{n} t_h \cos\left( 2 \, \pi \, h \, j \right) + \sum_{h=1}^{(n-1)/2}
\frac{h}{n} t_{n-h}\cos\left( 2 \, \pi \, h \, j \right) \right|\leq
\\ &\leq&\sum_{h=-(n-1)/2}^{(n-1)/2}
| t_h |\left(e^{\textrm{i}\frac{2 \, \pi \, j}{n}}
\right)^h \nonumber +\sum_{h=1}^{(n-1)/2}
\frac{h}{n} |t_h| + \sum_{h=1}^{(n-1)/2}
\frac{h}{n} |t_{n-h}|\leq\\ &\leq& \sum_{h=-\infty}^{+\infty}
| t_h |\left(e^{\textrm{i}\frac{2 \, \pi \, j}{n}}
\right)^h +\sum_{h=1}^{(n-1)/2}
\frac{h}{n} |t_h| + \sum_{h=(n+1)/2}^{n-1}
\frac{n-h}{n} |t_{h}|. \nonumber
\end{eqnarray}
Choose arbitrarily $\varepsilon > 0$.
Note that the first addend of the last term in (\ref{ruota}) tends to
$f \left(e^{\textrm{i}\frac{2 \, \pi \, j}{n}}
\right)$ as $n$ tends to $+\infty$, and hence,
without loss of generality, we can suppose that it
belongs to the interval $[f_{\rm{min}} - \varepsilon/6,
f_{\rm{max}} + \varepsilon/6]$ for $n$ sufficiently large.
By Lemma \ref{insp}, it is
\begin{eqnarray*}\label{firstmialc}
\lim_{n \to +\infty} \left(
\sum_{h=1}^{(n-1)/2}
\frac{h}{n} |t_h| + \sum_{h=(n+1)/2}^{n-1}
\frac{n-h}{n} |t_{h}|\right)=0.
\end{eqnarray*}
When $n$ is even, we get
\begin{eqnarray*}\label{ruotaa}   \nonumber
\left|\lambda_j^{(C_n(f))}\right|&=&\left|\sum_{h=0}^{n-1} c_h
\cos\left( 2 \, \pi \, h \, j \right) \right|=\left|
c_0 + 2 \sum_{h=1}^{n/2-1}c_h
\cos\left( 2 \, \pi \, h \, j \right) + (-1)^{n/2} 
c_{n/2}\right|=\\&=& \left|
t_0 + 2 \sum_{h=1}^{n/2-1} \cos\left( 2 \, \pi \, h \, j \right) 
t_h +(-1)^{n/2} t_{n/2} - \right.  \nonumber \\ 
 &-& \left.\sum_{h=1}^{n/2-1}
\frac{h}{n} t_h \cos\left( 2 \, \pi \, h \, j \right)+ \sum_{h=1}^{n/2-1}
\frac{h}{n} t_{n-h}\cos\left( 2 \, \pi \, h \, j \right)\right|\leq
\\ &\leq&\sum_{h=-n/2+1}^{n/2}
| t_h |\left(e^{\textrm{i}\frac{2 \, \pi \, j}{n}}
\right)^h \nonumber +\sum_{h=1}^{n/2-1}
\frac{h}{n} |t_h| + \sum_{h=1}^{n/2-1}
\frac{h}{n} |t_{n-h}|\\ &\leq& \sum_{h=-\infty}^{+\infty}
| t_h |\left(e^{\textrm{i}\frac{2 \, \pi \, j}{n}}
\right)^h  +\sum_{h=1}^{n/2-1}
\frac{h}{n} |t_h| + \sum_{h=n/2+1}^{n-1}
\frac{n-h}{n} |t_{h}|. \nonumber
\end{eqnarray*}
Thus, it is possible to repeat the same argument used in
the previous case, getting \ref{calculusestimate}.1).

Now we turn to \ref{calculusestimate}.2). 
From Theorem \ref{Theorem3.126} we obtain
\begin{eqnarray*} 
\lambda_j^{(B_n(f))}= \left\{
\begin{array}{ll}
\displaystyle{\sum_{h=0}^{n-1} \, b_h \cos 
\left(\dfrac{2 \, \pi \, j \, h}{n}\right)} & \text{ if } j \leq n/2,\\
\\ \displaystyle{- \sum_{h=0}^{n-1} \, b_h \cos 
\left(\dfrac{2 \, \pi \, (n-j) \, h}{n}\right)} & \text{ if } j > n/2.
\end{array}\right.
\end{eqnarray*}
So, without loss of generality, it is enough to prove 
\ref{calculusestimate}.2) for $j \leq n/2$. 

We first consider the case when $n$ is even. We get
\begin{eqnarray}\label{vwxyz}
\left|\lambda_j^{(B_n(f))}\right| &\leq& \left| \nonumber
\displaystyle{\sum_{h=0}^{n-1} \, b_h \cos 
\left(\dfrac{2 \, \pi \, j \, h}{n}\right)} 
\right| \leq \sum_{h=0}^{n-1}|b_h| = \\ &=& 
|b_0 | + \sum_{h=1}^{n/4-1}|b_{2h}| + 
\sum_{h=n/4 +1}^{n/2-1}|b_{2h}| + |b_{n/2}|+
\sum_{h=0}^{n/2-1}|b_{2h+1}|= \\&=&
I_1+I_2+I_3+I_4+I_5. \nonumber
\end{eqnarray}
So, in order to obtain \ref{calculusestimate}.2), it is enough 
to prove that each addend
of the last line of (\ref{vwxyz}) tends to $0$ as 
$n$ tends to $+\infty$. We get:
\begin{eqnarray}\label{I1estimate} 
I_1 &=& |b_0|\leq \sum_{k=1}^{n/2-1} \dfrac{4 \, k}{n^2} |t_{2k}|+
\sum_{k=1}^{n/2-1} \dfrac{4 \, k}{n^2} |t_{n-2k}|=
\\&=&\sum_{k=1}^{n/2-1} \dfrac{4 \, k}{n^2} |t_{2k}|+
\sum_{k=1}^{n/2-1} \dfrac{2 \, n-4 \, k}{n^2} |t_{2k}|
\leq \dfrac{4 \, n - 8}{n^2} \sum_{h=1}^{\infty}|t_h| 
\leq \dfrac{4 \, S}{n}, \nonumber
\end{eqnarray}
where $\displaystyle{S=\sum_{h=1}^{\infty}|t_h|}$.
From (\ref{I1estimate}) it follows that $I_1=|b_0|$
tends to $0$ 
as $n$ tends to $+ \infty$. Analogously it is possible to check 
that $I_4=|b_{n/2}|$ tends to $0$ 
as $n$ tends to $+ \infty$.

Now we estimate the term $I_2+I_3$. We first observe that
\begin{eqnarray}\label{otherestimate1}
& & 
\dfrac{1}{n^2} \nonumber
\sum_{h=1}^{n/4-1} (4 \, h - n) (|t_{2h}|+|t_{n-2h}|)+
\dfrac{1}{n^2}
\sum_{h=n/4+1}^{n/2-1} (4 \, h - n) (|t_{2h}|+|t_{n-2h}|)\leq
\\ &\leq&
\dfrac{1}{n^2}
\sum_{h=1}^{n/2-1} (4 \, h - n) (|t_{2h}|+|t_{n-2h}|) \leq 
\\&\leq& \dfrac{1}{n^2}
\sum_{h=1}^{n/2-1} 4 \, h |t_{2h}| + \nonumber
\nonumber \dfrac{1}{n^2}
\sum_{h=1}^{n/2-1} (4 \, h - n) |t_{n-2h}| =\\
&=&\dfrac{1}{n^2}
\sum_{h=1}^{n/2-1} 4 \, h |t_{2h}| +\dfrac{1}{n^2}
\sum_{h=1}^{n/2-1} (2 \, n-4 \, h ) |t_{2h}|. \nonumber
\end{eqnarray}
Arguing analogously as in (\ref{I1estimate}), it is possible to see
that the quantities at the first hand of
(\ref{otherestimate1}) tend to $0$ 
as $n$ tends to $+ \infty$.

Furthermore, we have
\begin{eqnarray}\label{os2}
\dfrac{2}{n} \sum_{h=1}^{n/2-1}\left( 
\sum_{k=1}^{h-1}\dfrac{2k}{n} |t_{2k}| \right)=
\sum_{k=1}^{n/2-2} \dfrac{4k}{n^2}\left( \dfrac{n}{2}-2-k\right)
|t_{2k}| \leq \sum_{k=1}^{n/2-2} \dfrac{2k}{n}|t_{2k}|,
\end{eqnarray}
\begin{eqnarray}\label{os3}
\dfrac{2}{n} \sum_{h=1}^{n/2-1}\left( 
\sum_{k=1}^{h-1}\dfrac{2k}{n} |t_{n-2k}| \right)=
\sum_{k=1}^{n/2-2} \dfrac{4k}{n^2}\left( \dfrac{n}{2}-2-k\right)
|t_{n-2k}| \leq \sum_{k=2}^{n/2-1} \dfrac{2k}{n}|t_{2k}|,
\end{eqnarray}
\begin{eqnarray}\label{os4}
\dfrac{2}{n} \sum_{h=1}^{n/2-1}\left( 
\sum_{k=1}^{n/2-h-1}\dfrac{2k}{n} |t_{2k}| \right)
\leq \sum_{k=1}^{n/2-2} \dfrac{2k}{n}|t_{2k}|,
\end{eqnarray}
and
\begin{eqnarray}\label{os5}
& & \dfrac{2}{n} \sum_{h=1}^{n/2-1}\left( \nonumber
\sum_{k=1}^{n/2-h-1}\dfrac{2k}{n} |t_{n-2k}| \right) \leq 
 \sum_{k=1}^{n/2-2} \dfrac{4 \, k}{n^2} \left(
\dfrac{n-2 \, k - 2}{2} \right) |t_{n - 2k}| = \\ &=&
\sum_{k=2}^{n/2-1} \dfrac{(n - 2 \, k) \, (2 \, k - 2)}{n^2}
|t_{2k}|
\leq  \sum_{k=2}^{n/2-1} \dfrac{2k}{n}|t_{2k}|.
\end{eqnarray}
Summing up (\ref{otherestimate1})-(\ref{os5}), from
(\ref{eveneven}) we obtain
\begin{eqnarray}\label{os6}
I_2 + I_3 &=&
\sum_{h=1}^{n/4-1}|b_{2h}| + \nonumber
\sum_{h=n/4 +1}^{n/2-1}|b_{2h}| \leq
\dfrac{1}{n^2}
\sum_{h=1}^{n/2-1} 4 \, h |t_{2h}| + \dfrac{1}{n^2}
\sum_{h=1}^{n/2-1} 4 \, h |t_{n-2h}| + \\  &+&
\sum_{k=1}^{n/2-2} \dfrac{4k}{n}|t_{2k}| +
\sum_{k=2}^{n/2-1} \dfrac{4k}{n}|t_{2k}|.
\end{eqnarray}

Thus, taking into account Lemma \ref{insp}, it is possible to
check that the terms at the right hand of
(\ref{os6}) tend to $0$ as $n$ tends to
$+\infty$.

Now we estimate the term $I_5$. One has
\begin{eqnarray*}\label{os3bis} & & \nonumber
\dfrac{2}{n} \sum_{h=0}^{n/2-1}\left( 
\sum_{k=0}^{h-1}\dfrac{2k+1}{n} |t_{2k+1}| \right)=
\sum_{k=0}^{n/2} \dfrac{2k+1}{n}|t_{2k+1}|= \\&=&
\sum_{k=0}^{n/2} \dfrac{2k}{n}|t_{2k+1}| + 
\sum_{k=1}^{n/2} \dfrac{1}{n}|t_{2k+1}| =J_1+J_2.
\end{eqnarray*}
Thanks to Lemma \ref{insp}, it is possible to
check that $J_1$ tends to $0$ as $n$ tends to $+\infty$.
Moreover, we have 
\begin{eqnarray}\label{159new}
0 \leq J_2 \leq \dfrac{S}{n},
\end{eqnarray}
and hence $I_4$ tends to $0$ as $n$ tends to $+\infty$.
Analogously as in the previous case, it is possible to prove that 
\begin{eqnarray}\label{os6bis}
I_5&=&\sum_{k=0}^{n/2-1}|b_{2k+1}|\nonumber
\leq
\dfrac{1}{n^2}
\sum_{k=0}^{n/2-1} 2\, (2\, k + 1) |t_{2k+1}| + \\&+&
 \dfrac{2}{n}
\sum_{k=1}^{n/2-2}  \left(\dfrac{n}{2} -1 -k\right) 
\left( \dfrac{2 \, k +1}{n}\right)|t_{n-2k-1}|\leq \\  &\leq&
\sum_{k=1}^{n/2-2} \dfrac{2\, (2\, k + 1) }{n}|t_{2k+1}| +
\sum_{k=2}^{n/2-1} \dfrac{2\, (2\, k + 1) }{n}|t_{2k+1}|.
\nonumber
\end{eqnarray}
By virtue of Lemma \ref{insp} and (\ref{159new}), we
get that $I_5$ tends to $0$ as $n$ tends to $+\infty$. Therefore,
all addends of the right hand of (\ref{vwxyz}) tend 
to $0$ as $n$ tends to $+ \infty$. Thus, (\ref{lambd})
follows from (\ref{vwxyz}), (\ref{I1estimate}), (\ref{os6})
and (\ref{os6bis}).

When $n$ is odd, it is possible 
to proceed analogously as in previous case.
This proves \ref{calculusestimate}.1).
\vspace{3mm}

Now we turn to \ref{calculusestimate}.2), that is 
we prove that the spectrum of $(G_n(f))^{-1} \, T_n(f)$ 
is clustered around $1$. Since
$(G_n(f))^{-1} (T_n(f) - G_n(f)) =(G_n(f))^{-1} T_n(f) - I_n$,
where $I_n$ is the identity matrix, it is enough
%
to check that the
eigenvalues of $(G_n(f))^{-1} (T_n (f)- G_n(f))$ are clustered 
around $0$. 

Choose arbitrarily $\varepsilon > 0$. 
Since $f$ belongs 
to the Wiener class, there exists a positive integer $n_0
=n_0(\varepsilon)$ such that 
$$\sum_{j=n_0+1}^{\infty}|t_j| \leq \varepsilon.$$
Proceeding similarly as in the proof of 
\cite[Theorem 3 (ii)]{BINIFAVATI1993}, we get
\begin{eqnarray*}
T_n(f) - G_n(f)= T_n(f) - C_n(f) - B_n(f) =
E_n^{(n_0)} + W_n^{(n_0)} + Z_n^{(n_0)},
\end{eqnarray*}
where $E_n^{(n_0)}$, $W_n^{(n_0)}$ and $Z_n^{(n_0)}$ are 
suitable matrices such that 
$E_n^{(n_0)}$ and $W_n^{(n_0)}$ agree
with the 
$(n - n_0) \times (n - n_0)$ leading principal
submatrices of $T_n(f)-C_n(f)$ and $B_n(f)$, respectively.
We have:
\begin{eqnarray}\label{estimates}
\nonumber
{\rm{rank}}(E_n^{(n_0)}) &\leq& 2 \, n_0; \\
\|W_n^{(n_0)}\|_1 &\leq& \frac{2}{n} \sum_{k=1}^{n-n_0-1}
k \, |t_{n-k}-t_k| \leq  \frac{2}{n} \sum_{k=1}^{n_0}
k \, |t_k| + 4 \sum_{k=n_0+1}^{\infty} |t_k|; \\  
\|Z_n^{(n_0)}\|_1 &\leq& \sum_{h=0}^{n-1} |b_h|,
\nonumber
\end{eqnarray}
where the symbol 
$\|\cdot\|_1$ denotes the $1$-norm of the involved matrix.
Let $n_1 > n_0 $ be a positive integer with 
\begin{eqnarray}\label{inspestimates}
\dfrac{1}{n_1}\sum_{k=1}^{n_0} k \, |t_k| \leq 
\varepsilon \,  \, \text{     and     } \,\,
\sum_{h=0}^{n-1} |b_h| \leq 
\varepsilon.
\end{eqnarray} Note that such an $n_1$ does exist,
thanks to Lemma \ref{insp} and since all terms of
(\ref{vwxyz}) tend to $0$ as $n$ tends to $+ \infty$.
From (\ref{estimates}) and (\ref{inspestimates}) 
it follows that 
\begin{eqnarray}\label{int}
\|W_n^{(n_0)} +W_n^{(n_0)}\|_1
\leq \|W_n^{(n_0)}\|_1 + \|Z_n^{(n_0)}\|_1
\leq 8 \varepsilon.
\end{eqnarray}
From (\ref{int}) and the Cauchy interlace theorem
(see, e.g., \cite{WILKINSON}) we deduce that the eigenvalues 
of $T_n(f)-G_n(f)$ are clustered around $0$,
with the exception of at most $k= 2 \, n_0$ of them. 
By virtue of the Courant-Fisher minimax characterization 
of the matrix $(G_n(f))^{-1} \,( T_n(f)-G_n(f) )$
(see, e.g., \cite{WILKINSON}), we obtain
\begin{eqnarray} \label{fin}
\lambda_j^{(G_n(f))^{-1} \, (T_n(f)-G_n(f) )} \leq
\dfrac{\lambda_j^{(T_n(f)-G_n(f) )}}{f_{\rm{min}}}
\end{eqnarray}
for $n$ large enough. From (\ref{fin}) we deduce 
that the spectrum of $(G_n(f))^{-1} \,( T_n(f)-G_n(f) )$
is clustered around $0$, namely for every 
$\varepsilon > 0$ there are $k$, $n_1 \in \mathbb{N}$
with the property that for each $\varepsilon > 0$ 
the number of eigenvalues 
$\lambda_j^{(G_n(f))^{-1} \, T_n(f)}$ such that 
$\left|\lambda_j^{(G_n(f))^{-1} \, T_n(f)}-1\right|>
\varepsilon$ is at most equal to $k$.
\end{proof}
\section{Conclusions}
We analyzed a new class of 
simultaneously diagonalizable real matrices,
the $\gamma$-ma\-tri\-ces. Such a class
includes both symmetric circulant matrices and a subclass of
reverse circulant matrices.
We dealt with both spectral and structural
properties of this family of matrices. 
Successively, we defined some algorithms 
for fast computation
of the product between a $\gamma$-matrix
and a real vector 
and between two $\gamma$-matrices, and we proved
that the computational cost of a multiplication
between a $\gamma$-matrix and a real vector is of at most
$\frac74 \, n \, \log_2 n+o( n \, \log_2 n)$ additions and
$\frac12 \, n \, \log_2 n+o( n \, \log_2 n)$ multiplications, 
and the computational cost of a multiplication
between two $\gamma$-matrices is at most
$\frac92 \, n \, \log_2 n+o( n \, \log_2 n)$ additions and
$\frac32 \, n \, \log_2 n+o( n \, \log_2 n)$ multiplications.
Finally, we gave a technique of approximating a real symmetric
Toeplitz matrix by a $\gamma$-matrix, and we proved how
the eigenvalues of the preconditioned matrix are clustered 
around $0$ with the exception of at most a finite number 
of terms.


%
%


\end{document}